\documentclass[reqno]{amsart}

\usepackage{amssymb}
\usepackage{mathtools}
\usepackage{a4wide,amsmath}
\usepackage{mathrsfs}
\usepackage{amsthm}
\numberwithin{equation}{section}
\numberwithin{figure}{section}
\numberwithin{table}{section}
\usepackage{bbm}
\usepackage{enumerate}
\usepackage[utf8]{inputenc}
\usepackage{hyperref}
\hypersetup{hidelinks}
\usepackage[textsize=tiny]{todonotes} 

\long\def\MSC#1\EndMSC{\def\arg{#1}\ifx\arg\empty\relax\else
     {\narrower\noindent%
{2010 Mathematics Subject Classification}: #1\\} \fi}
\long\def\PACS#1\EndPACS{\def\arg{#1}\ifx\arg\empty\relax\else
     {\narrower\noindent%
{PACS numbers}: #1}\fi}
\long\def\KEY#1\EndKEY{\def\arg{#1}\ifx\arg\empty\relax\else
	{\narrower\noindent%
Keywords: #1\\}\fi}

\theoremstyle{plain}
\newtheorem{theorem}{Theorem}[section]
\newtheorem{lemma}[theorem]{Lemma}
\newtheorem{proposition}[theorem]{Proposition}
\newtheorem{corollary}[theorem]{Corollary}
\theoremstyle{definition}
\newtheorem{definition}[theorem]{Definition}
\theoremstyle{remark}
\newtheorem{remark}[theorem]{Remark}
\newtheorem{example}[theorem]{Example}

\newcommand{\N}{\mathbb{N}}
\newcommand{\R}{\mathbb{R}}
\newcommand{\C}{\mathbb{C}}

\DeclareMathOperator*{\essinf}{ess\,inf}
\DeclareMathOperator*{\esssup}{ess\,sup}

\begin{document}

\title[On regularity of the logarithmic forward map of EIT]{On regularity of the logarithmic forward map of electrical impedance tomography}

\author[H.~Garde]{Henrik Garde}
\address[H.~Garde]{Department of Mathematical Sciences, Aalborg University, Skjernvej 4A, 9220 Aalborg, Denmark.}
\email{henrik@math.aau.dk}

\author[N.~Hyv\"onen]{Nuutti Hyv\"onen}
\address[N.~Hyv\"onen]{Department of Mathematics and Systems Analysis, Aalto University, P.O. Box~11100, 02150 Espoo, Finland.}
\email{nuutti.hyvonen@aalto.fi}

\author[T.~Kuutela]{Topi Kuutela}
\address[T.~Kuutela]{Department of Mathematics and Systems Analysis, Aalto University, P.O.~Box 11100, 02150 Espoo, Finland.}
\email{topi.kuutela@aalto.fi}

\begin{abstract}
 This work considers properties of the logarithm of the Neumann-to-Dirichlet boundary map for the conductivity equation in a Lipschitz domain. It is shown that the mapping from the (logarithm of) the conductivity, i.e.~the (logarithm of) the coefficient in the divergence term of the studied elliptic partial differential equation, to the logarithm of the Neumann-to-Dirichlet map is continuously Fr\'echet differentiable between natural topologies. Moreover, for any essentially bounded perturbation of the conductivity, the Fr\'echet derivative defines a bounded linear operator on the space of square integrable functions living on the domain boundary, although the logarithm of the Neumann-to-Dirichlet map itself is unbounded in that topology. In particular, it follows from the fundamental theorem of calculus that the difference between the logarithms of any two Neumann-to-Dirichlet maps is always bounded on the space of square integrable functions. All aforementioned results also hold if the Neumann-to-Dirichlet boundary map is replaced by its inverse,~i.e.~the Dirichlet-to-Neumann map.
\end{abstract}

\maketitle

\KEY
Neumann-to-Dirichlet map, 
Fr\'echet derivative, 
logarithm, 
functional calculus, 
electrical impedance tomography.
\EndKEY

\MSC
35J15, 
46T20, 
47A60, 
35R30. 
\EndMSC

\section{Introduction}
This work is motivated by {\em electrical impedance tomography} (EIT),~i.e.,~the imaging modality whose aim is to reconstruct (useful information about) the conductivity inside a physical body from boundary measurements of current and voltage. The idealized mathematical model for EIT is to determine the strictly positive and bounded coefficient $\sigma: \Omega \to \R$ in the elliptic conductivity equation
\begin{equation}
\label{eq:intro}
\nabla \cdot (\sigma \nabla u) = 0 \qquad \text{in } \Omega
\end{equation}
from the Cauchy data of all its solutions on the boundary of the domain $\Omega \subset \R^d$, $d \geq 2$. In this paper, we employ this ideal model and assume the available measurement is the {\em Neumann-to-Dirichlet} (ND) boundary map associated to~\eqref{eq:intro}, although all practical setups for EIT actually employ a finite number of contact electrodes, resulting in a finite-dimensional measurement (cf.~\cite{Cheng89,Somersalo92}). However, it would also be possible to formulate our main ideas for realistic electrode models~(cf.~\cite{Hyvonen18}). For more information on practical EIT as well as on the related theoretical uniqueness and stability results, we refer to the review papers~\cite{Borcea02, Borcea2002,Cheney99,Uhlmann09} and the references therein.

The reconstruction task of EIT, as any other nonlinear inverse problem, can be straightforwardly tackled via regularized least squares minimization, that is, by iteratively linearizing the dependence of the data on the unknown and solving the resulting illposed linear problems by resorting to a suitable regularization technique. In EIT, it is also possible to obtain useful information about the unknown by only taking a single linearization step~\cite{Adler09,Cheney90,Harrach10}. The success of such straightforward approaches definitely depends on the degree of nonlinearity in the forward map that takes the unknown to the data,~i.e.~on the linearization error. On the other hand, the nonlinearity of the forward map can be altered by choosing different parametrizations for the unknown and the data.

Such an idea was tested for EIT in \cite{Hyvonen18}, where it was numerically demonstrated that the {\em completely logarithmic forward map} of EIT, taking the logarithm of the conductivity to the logarithm of the ND map, is significantly less nonlinear than, say, the standard forward map that sends the conductivity to the ND map itself.
To be slightly more precise, the mean relative linearization errors around the unit conductivity were computed over certain random samples of 50\,000 conductivities in the unit disk with different parametrizations for the forward map of EIT, and these mean errors were found to be approximately an order of magnitude smaller for the completely logarithmic forward map than for the standard one. This lower degree of nonlinearity was also observed with the complete electrode model (see \cite{Cheng89,Somersalo92}) as well as in the mean $L^2(\Omega)$ reconstruction errors for a simple one-step reconstruction algorithm. What is more, the `almost linearity' of the completely logarithmic forward map can actually be explicitly characterized in some simple geometries~\cite[Examples~2 \& 3]{Hyvonen18}. It should be noted, however, that some other transformation  could well lead to an even more advantageous parametrization for the forward map of EIT.

The studies in~\cite{Hyvonen18} were mainly based on finite-dimensional numerical approximations.
In particular, the Fr\'echet differentiability of the infinite-dimensional completely logarithmic forward map of EIT was not established and no actual mathematical proof for its low degree of nonlinearity was presented. The main goal of this work is to fix the first of these two imperfections.
To be more precise, we prove that the completely logarithmic forward map is continuously Fr\'echet differentiable from $L^\infty(\Omega)$ to the space of bounded linear operators between the mean-free Sobolev spaces $H^{\epsilon}_\diamond(\partial \Omega)$ and $H^{-\epsilon}_\diamond(\partial \Omega)$ for any $\epsilon > 0$. This is not an obvious result because the eigenvalues of a ND map accumulate at the origin and those of its logarithm at minus infinity.

Although it is natural to consider the logarithm of an ND map as an operator from $H^{\epsilon}_\diamond(\partial \Omega)$ to $H^{-\epsilon}_\diamond(\partial \Omega)$ because it is not bounded on the space of mean-free square integrable functions $L^2_\diamond(\partial \Omega)$, it turns out that the corresponding Fr\'echet derivative is more regular and defines a bounded linear operator on $L^2_\diamond(\partial\Omega)$ for any essentially bounded perturbation of the (log-)conductivity. In particular, it follows from the fundamental theorem of calculus that the difference between the logarithms of any two ND maps is always bounded on  $L^2_\diamond(\partial\Omega)$.
We want to emphasize that this slight increase in regularity when taking the difference
holds for any two log-conductivities in $L^\infty(\Omega)$ without any extra assumptions on their (common) behavior at or close to $\partial \Omega$; such assumptions are needed for the difference of two ND maps to exhibit higher regularity than either of the maps on their own (cf.,~e.g.,~\cite{Lee1989}). Loosely speaking, the `singularity' preventing the logarithms of ND maps from mapping $L^2_\diamond(\partial\Omega)$ to itself is the same for all conductivities, and it thus disappears when subtracting any two of such logarithms.

We present all our results for the ND map as it is more suitable for numerical studies than its inverse, the {\em Dirichlet-to-Neumann} (DN) map. However, our main theorems could as well be formulated for the DN map because the logarithms of the ND and DN maps for a given conductivity only differ by a change of sign.

This article is organized as follows. This introduction is first completed by reviewing the employed notation and terminology. The mathematical setting is formally introduced and the main results are formulated in Section~\ref{sec:setting}. Subsequently, Section~\ref{sec:shifted} moves the spectrum of the ND map by $\tau > 0$ to the right from the origin and proves Fr\'echet differentiability and other auxiliary results for the associated shifted logarithmic isomorphism. Finally, the main results are proven in Section~\ref{sec:convergence} by letting $\tau>0$ tend to zero in a controlled manner. The paper is concluded with two appendices considering the Fr\'echet derivatives of the ND map and equivalent norms for the mean-free Sobolev spaces $H^{r}_\diamond(\partial \Omega)$ with $r \in [-\tfrac{1}{2}, \tfrac{1}{2}]$. These equivalent norms form an essential tool for our analysis.

\subsection{Notation and terminology} \label{sec:notation}

We denote by $\mathscr{L}(X,Y)$ the space of bounded linear operators between Banach spaces $X$ and $Y$, and introduce the shorthand notation $\mathscr{L}(X) := \mathscr{L}(X,X)$. If $X$ is a Hilbert space, then $\mathscr{L}_{\rm sa}(X) \subset \mathscr{L}(X)$ denotes the closed subspace consisting of the bounded self-adjoint operators. 

Let $\Omega \subset \R^d$, $d \geq 2$, be a bounded domain with a Lipschitz boundary $\partial \Omega$. The bracket $\langle \cdot, \cdot \rangle \colon H^{-r}(\partial \Omega)\times H^{r}(\partial \Omega) \to \C$ denotes the \emph{sesquilinear} dual pairing on $\partial \Omega$ with an interpretation as an extension of the $L^2(\partial \Omega)$ inner product. 

For $r\in[-\tfrac{1}{2},\tfrac{1}{2}]$, we define the mean-free subspaces of $H^r(\partial\Omega)$ as
\begin{equation*}
H_\diamond^{r}(\partial \Omega) := \left\{ v \in H^{r}(\partial \Omega) \, | \,  \langle 1, v\rangle = 0 \right\}.
\end{equation*}
Since $H_\diamond^r(\partial\Omega) \subset H^r(\partial\Omega)$, it follows that $H^{-r}_\diamond(\partial\Omega) \subset H^{-r}(\partial\Omega) =  (H^r(\partial\Omega))' \subseteq (H_\diamond^r(\partial\Omega))'$, where the latter inclusion is not an embedding as it is not injective. On the other hand, if $f\in (H^r_\diamond(\partial\Omega))'$, then we may define its extension by zero via
\begin{equation*}
\langle\tilde{f},g\rangle = \begin{cases}
\langle f,g\rangle, & \quad g\in H^r_\diamond(\partial\Omega), \\[1mm]
0, & \quad g\in \textup{span}(1).
\end{cases}
\end{equation*} 
Obviously, $\tilde{f}\in (H^r(\partial\Omega))' = H^{-r}(\partial\Omega)$ is well defined and satisfies $\langle \tilde{f},1 \rangle = 0$, i.e., $\tilde{f} \in H^{-r}_{\diamond}(\partial\Omega)$. Identifying $f$ with $\tilde{f}$ gives a natural isometric isomorphism between $(H^r_\diamond(\partial\Omega))'$ and $H^{-r}_\diamond(\partial\Omega)$.

We denote by $T^*\in \mathscr{L}(H_\diamond^{-r}(\partial\Omega),H_\diamond^r(\partial\Omega))$ the unique \emph{dual operator} of a bounded linear map $T\in \mathscr{L}(H_\diamond^{-r}(\partial\Omega),H_\diamond^r(\partial\Omega))$ with respect to the dual bracket. In particular, for all $f,g\in H_\diamond^{-r}(\partial\Omega)$ it holds
\begin{equation*}
\langle f,Tg \rangle = \overline{\langle g, T^* f\rangle}.
\end{equation*}
If $T=T^*$, we call $T$ \emph{symmetric} with respect to the dual bracket.

\section{The setting and main results}
\label{sec:setting}
In this section, we first recall the definition of the ND map for the conductivity equation together with some of its basic properties. Subsequently, we define the logarithmic ND operator, introduce the completely logarithmic forward map of EIT, and finally state the main results of this work.

\subsection{Neumann-to-Dirichlet operator and its derivatives}
\label{sec:param}
Let us consider the Neumann boundary value problem
\begin{equation}
\label{eq:eitcont}
\begin{split}
\nabla \cdot (\sigma \nabla u) &= 0 \qquad \text{in } \Omega, \\
\sigma \frac{\partial u}{\partial \nu} &= f \qquad \text{on } \partial \Omega,
\end{split}
\end{equation}
where $\Omega \subset \R^d$, $d \geq 2$, is a bounded domain with a Lipschitz boundary $\partial \Omega$. The electrical conductivity $\sigma \in L_+^\infty(\Omega)$ is real-valued and isotropic, but $f \in H_\diamond^{-1/2}(\partial \Omega)$ is in general complex-valued.
The conductivity coefficient $\sigma$ is bounded from below by a positive constant, that is,
\begin{equation*}
L_+^\infty(\Omega) := \left\{ \varsigma \in L^\infty(\Omega; \R) \, | \,  \essinf \varsigma > 0  \right\}.
\end{equation*}
Note that apart from $L^\infty(\Omega) := L^\infty(\Omega; \R)$, the multiplier field for all employed function spaces is~$\C$.

The variational form of the Neumann problem \eqref{eq:eitcont} is to find $u \in H^1(\Omega)$ such that
\begin{equation}
\label{eq:varcont}
\int_\Omega \sigma \nabla u \cdot \nabla\overline{v} \, {\rm d} x = \big\langle f, v|_{\partial \Omega} \big\rangle
\end{equation}
holds for all $v \in H^1(\Omega)$.
The standard theory for elliptic partial differential equations reveals that there exists a unique solution to \eqref{eq:varcont} in the quotient space $H^1(\Omega) / \mathbb{C}$  for any given current density $f \in H^{-1/2}_\diamond(\partial \Omega)$. In particular, there is a unique mean-free boundary potential
\begin{equation}
\label{eq:Udef}
U := u|_{\partial \Omega} \in H_\diamond^{1/2}(\partial \Omega)
\end{equation}
that depends linearly and boundedly on the corresponding $f \in H^{-1/2}_\diamond(\partial \Omega)$. To be more precise,
\begin{equation}
\label{eq:Ubnd}
\| U \|_{H^{1/2}(\partial \Omega)} \leq C(\Omega) \| u \|_{H^{1}(\Omega)/\C}
\leq \frac{C(\Omega)}{\essinf{\sigma}} \| f \|_{H^{-1/2}(\partial \Omega)},
\end{equation}
as easily deduced using the Lax--Milgram lemma, trace theorem, and Poincar\'e inequality~\cite{Grisvard85}.

The linear map $f \mapsto U$, which obviously depends on $\sigma$, is called the ND operator and denoted by $\Lambda(\sigma)$. For any given $\sigma \in L_+^\infty(\Omega)$, the mapping
\begin{equation*}
  \Lambda(\sigma) \colon
  \left\{
  \begin{array}{l}
    f \mapsto U, \\[1mm]
    H_\diamond^{-1/2}(\partial \Omega) \to H_\diamond^{1/2}(\partial \Omega)
    \end{array}
    \right.
  \end{equation*}
is a symmetric linear isomorphism. Moreover $\Lambda(\sigma)$ is positive,
\begin{equation*}
\langle f, \Lambda(\sigma) f \rangle \geq c \| f \|_{H^{-1/2}(\partial \Omega)}^2 \qquad {\rm for} \ {\rm all} \ f \in H^{-1/2}_{\diamond}(\partial \Omega) \ {\rm and} \ {\rm some} \ c>0, 
\end{equation*}
as can be easily deduced from \eqref{eq:varcont} and Neumann trace theorems for those elements of $H^1(\Omega)/\mathbb{C}$ for which the range of $\nabla \cdot (\sigma \nabla(\cdot))$ is a subspace of $L^2(\Omega)$ (cf.,~e.g.,~\cite[p.~381, Lemma 1]{Dautray88}).

It follows from \eqref{eq:varcont} that considering $\Lambda(\sigma)$ as an element of $\mathscr{L}(L^2_\diamond(\partial\Omega))$ makes it self-adjoint as well as compact due to the compact and dense embeddings $H_\diamond^{1/2}(\partial \Omega) \hookrightarrow L^2_{\diamond}(\partial \Omega) \hookrightarrow H_\diamond^{-1/2}(\partial \Omega)$,
which inherit their properties from  the standard case $H^{1/2}(\partial \Omega) \hookrightarrow L^2(\partial \Omega) \hookrightarrow H^{-1/2}(\partial \Omega)$,~e.g.,~via taking intersections with $H_\diamond^{-1/2}(\partial \Omega)$.
In particular,  $\Lambda(\sigma)$  admits a spectral decomposition
\begin{equation}
\label{eq:spectral}
\Lambda(\sigma) f = \sum_{k=1}^\infty \lambda_k \langle f, \phi_k\rangle \, \phi_k,
\end{equation}
where the eigenvalues satisfy $\lambda_{k} \geq \lambda_{k+1}$ and $\R_+ \ni \lambda_k \to 0$ as $k \to \infty$, and the corresponding eigenfunctions $\{\phi_k\}_{k=1}^\infty \subset H_\diamond^{1/2}(\partial \Omega)$ form an orthonormal basis for $L_\diamond^2(\partial \Omega)$. Observe that the representation \eqref{eq:spectral} holds for all $f \in H_\diamond^{-1/2}(\partial \Omega)$ by boundedness of $\Lambda(\sigma) \colon H_\diamond^{-1/2}(\partial \Omega) \to H_\diamond^{1/2}(\partial \Omega)$ and the density of the embedding $L^2_{\diamond}(\partial \Omega) \hookrightarrow H_\diamond^{-1/2}(\partial \Omega)$.

The nonlinear mapping
\begin{equation}
\label{eq:lambda}
\Lambda \colon \sigma \mapsto \Lambda(\sigma), \quad L_+^\infty(\Omega) \to \mathscr{L}(H_\diamond^{-1/2}(\partial \Omega), H_\diamond^{1/2}(\partial \Omega))
\end{equation}
is in this work called the \emph{standard forward operator} of EIT. The idealized inverse problem of EIT is to find $\sigma$ from the knowledge of $\Lambda(\sigma)$.  It is well known that the map $\sigma \mapsto \Lambda(\sigma)$ is infinitely times continuously Fr\'echet differentiable. In particular, the first and second derivatives, 
\begin{align*}
D\Lambda(\sigma;\eta) &\colon H_\diamond^{-1/2}(\partial \Omega) \to  H_\diamond^{1/2}(\partial \Omega), \\
D^2\!\Lambda(\sigma;\eta,\xi) &\colon H_\diamond^{-1/2}(\partial \Omega) \to  H_\diamond^{1/2}(\partial \Omega),
\end{align*}
are continuous with respect to $\sigma \in L^\infty_+(\Omega)$, symmetric, and depend (bi)linearly and boundedly on the perturbations $\eta,\xi \in L^\infty(\Omega)$ with respect to the topology of  $\mathscr{L}(H_\diamond^{-1/2}(\partial \Omega), H_\diamond^{1/2}(\partial \Omega))$. To be more precise,
\begin{align}
\label{eq:FrDB}  
\| D\Lambda(\sigma;\, \cdot \,) \|_{\mathscr{L}(L^\infty(\Omega),\mathscr{L}(H_\diamond^{-1/2}(\partial \Omega), H_\diamond^{1/2}(\partial \Omega)))} &\leq \frac{C}{\essinf{\sigma}^2}, \\[1mm]
\label{eq:FrDB2}
\| D^2\!\Lambda(\sigma;\, \cdot \,, \, \cdot \,) \|_{\mathscr{L}(L^\infty(\Omega)^2,\mathscr{L}(H_\diamond^{-1/2}(\partial \Omega), H_\diamond^{1/2}(\partial \Omega)))} &\leq \frac{C}{\essinf{\sigma}^3},
\end{align}
where $C = C(\Omega)>0$ does not depend on $\sigma$. For the sake of completeness, we have included the precise definitions of $D\Lambda(\sigma;\eta)$ and $D^2\!\Lambda(\sigma;\eta,\xi)$ and the proofs of \eqref{eq:FrDB} and \eqref{eq:FrDB2} in Appendix~\ref{app:derivatives}. 

\subsection{Logarithmic forward operator and the main results}
\label{sec:log}

Using the spectral decomposition~\eqref{eq:spectral}, the logarithm of the ND operator can be defined as
\begin{equation}
\label{eq:logspectral}
\log\!\Lambda(\sigma) \colon f \mapsto \sum_{k=1}^\infty \log(\lambda_k) \langle f,\phi_k \rangle \, \phi_k,
\end{equation}
where $\log$ denotes the principal branch of the natural logarithm. As demonstrated in~\cite{Hyvonen18}, if one defines the domain of $\log\!\Lambda(\sigma)$ to be
\begin{equation*}
\mathcal{D}\big(\!\log\!\Lambda(\sigma)\big)  = \bigg\{g \in L^2_\diamond(\partial \Omega) \ \Big| \  
\sum_{k=1}^\infty  \log^2(\lambda_k) |\langle g,\phi_k \rangle|^2 < \infty \bigg\},
\end{equation*}
it becomes a self-adjoint unbounded operator on $L_\diamond^2(\partial \Omega)$ for any $\sigma \in L^\infty_+(\Omega)$. However,  $\log\!\Lambda(\sigma)$ can also be interpreted as a symmetric compact operator
\begin{equation*}
\log\!\Lambda(\sigma): H^\epsilon_\diamond(\partial \Omega)  \to H^{-\epsilon}_\diamond(\partial \Omega), \qquad \epsilon > 0,
\end{equation*}
see \cite[Corollary~1]{Hyvonen18}.

Following \cite{Hyvonen18}, we now introduce the completely logarithmic forward map of EIT.
\begin{definition}
	The {\em completely logarithmic forward map} is defined via
	\begin{equation}
	\label{eq:log_for}
	L:  \kappa \mapsto \log\!\Lambda({\rm e}^\kappa), \quad L^\infty(\Omega) \to \mathscr{L}(H^\epsilon_\diamond(\partial \Omega), H^{-\epsilon}_\diamond(\partial \Omega))
	\end{equation}
	for any fixed $\epsilon >0$.
\end{definition}

Now we are finally ready to present the main results of this work.

\begin{theorem}
	\label{thm:main}
	The completely logarithmic forward map of EIT 
	\begin{equation*}
	L: L^\infty(\Omega) \to \mathscr{L}(H^\epsilon_\diamond(\partial \Omega), H^{-\epsilon}_\diamond(\partial \Omega)), \qquad \epsilon > 0,
	\end{equation*}
	is Fr\'echet differentiable. The corresponding Fr\'echet derivative
	\begin{equation*}
	DL(\kappa; \, \cdot \,) \in \mathscr{L}\big(L^\infty(\Omega), \mathscr{L}(L^2_\diamond(\partial \Omega))\big)  \subset \mathscr{L}\big(L^\infty(\Omega), \mathscr{L}(H^\epsilon_\diamond(\partial \Omega), H^{-\epsilon}_\diamond(\partial \Omega))\big)
	\end{equation*}
	depends continuously on the log-conductivity $\kappa \in L^\infty(\Omega)$. More precisely, for any  $\kappa_1,\kappa_2\in L^\infty(\Omega)$ lying inside the origin-centered ball of radius $R>0$ in the topology of $L^\infty(\Omega)$, it holds
	\begin{equation}
          \label{eq:maincont}
	\| DL(\kappa_2,\, \cdot \,) - DL(\kappa_1,\, \cdot \,) \|_{\mathscr{L}(L^\infty(\Omega),\mathscr{L}(L^2_\diamond(\partial\Omega)))} \leq C \| {\rm e}^{\kappa_2} - {\rm e}^{\kappa_1} \|_{L^\infty(\Omega)},
	\end{equation}
	where $C>0$ only depends on $\Omega$ and $R$.
\end{theorem}

Although $L(\kappa) \notin \mathscr{L}(L^2_\diamond(\partial \Omega))$ for all $\kappa \in L^\infty(\Omega)$, it follows straightforwardly from Theorem~\ref{thm:main} that this regularity issue disappears for a {\em relative} completely logarithmic forward map. 

\begin{corollary}
	\label{cor:main}
	For any $\kappa, \kappa_0 \in L^\infty(\Omega)$, the operator $L(\kappa) - L(\kappa_0)$ continuously extends to a self-adjoint operator in $\mathscr{L}(L^2_\diamond(\partial\Omega))$. In particular, for a fixed $\kappa_0\in L^\infty(\Omega)$, 
	\begin{equation*}
	\kappa\mapsto L(\kappa) - L(\kappa_0)
	\end{equation*}
	is continuously Fr\'echet differentiable as a map $L^\infty(\Omega)\to \mathscr{L}(L^2_\diamond(\partial \Omega))$, with the derivative $DL(\kappa,\, \cdot \,)$. Moreover, for any  $\kappa_1,\kappa_2\in L^\infty(\Omega)$ lying inside the origin-centered ball of radius $R>0$ in the topology of $L^\infty(\Omega)$, it holds
	\begin{equation} \label{eq:Llocaluniformcont}
	\| L(\kappa_2) - L(\kappa_1) \|_{\mathscr{L}(L^2_\diamond(\partial\Omega))} \leq C \| {\rm e}^{\kappa_2} - {\rm e}^{\kappa_1} \|_{L^\infty(\Omega)},
	\end{equation}
	where $C>0$ only depends on $\Omega$ and $R$.
\end{corollary}

The following example sheds light on Corollary~\ref{cor:main} in case of homogeneous conductivities. The reader is also encouraged to consult \cite[Example~3]{Hyvonen18}.

\begin{example}
  Consider two positive {\em constant} conductivities $\sigma, \sigma_0 \in \R_+$ in $\Omega$ and denote the corresponding log-conductivities by $\kappa, \kappa_0 \in \R$, respectively. If $\Lambda(\sigma)$ obeys the spectral decomposition \eqref{eq:spectral}, then via a simple scaling argument (see,~e.g.,~\cite[Example~1]{Hyvonen18}), 
$$
\Lambda(\sigma_0)f = \sum_{k=1}^\infty \frac{\sigma}{\sigma_0} \lambda_k \langle f, \phi_k\rangle \, \phi_k, \qquad f \in L^2_\diamond(\partial \Omega).
$$
Hence, for any $f \in H^{\epsilon}_{\diamond}(\partial \Omega)$ with $\epsilon > 0$,
$$
\big( L(\kappa) - L(\kappa_0) \big) f = \big(\log \! \Lambda({\rm e}^{\kappa}) - \log \! \Lambda({\rm e}^{\kappa_0}) \big) f = \sum_{k=1}^\infty (\kappa_0 - \kappa) \langle f, \phi_k\rangle \, \phi_k = (\kappa_0 - \kappa) f.
$$
In other words, $L(\kappa) - L(\kappa_0)$ extends to $\mathscr{L}(L^2_\diamond(\partial \Omega))$ as $(\kappa_0 - \kappa)I$, where $I$ is the identity map. 
\end{example}

We also obtain the following representation for $DL$, generalizing \cite[Theorem~1]{Hyvonen18} to infinite dimensions.
\begin{corollary} \label{cor:DLformula}
Let  $\kappa,\eta\in L^\infty(\Omega)$ be arbitrary and denote by $\{\lambda_k, \phi_k\}_{k\in\mathbb{N}}$ a normalized eigensystem of $\Lambda({\rm e}^\kappa)$. Then for any $f\in L^2_\diamond(\partial\Omega)$, 
	\begin{equation*} 
	DL(\kappa;\eta)f = \sum_{j=1}^\infty \sum_{k=1}^\infty c_{j,k} \, \langle f,\phi_k\rangle \big\langle D\Lambda({\rm e}^\kappa; \eta{\rm e}^\kappa)\phi_k,\phi_j\big\rangle \, \phi_j,
	\end{equation*}
	where 
	\begin{equation*}
	c_{j,k} := \begin{cases}
	\frac{\log(\lambda_j)-\log(\lambda_k)}{\lambda_j-\lambda_k},  & \quad \lambda_j\neq \lambda_k, \\[1mm]
	\frac{1}{\lambda_j}, & \quad \lambda_j=\lambda_k.
	\end{cases}
	\end{equation*}	
\end{corollary}

The rest of this paper aims at proving the above results and providing further insight on their natural generalizations;
the final proofs of Theorem~\ref{thm:main}, Corollary \ref{cor:main}, and Corollary \ref{cor:DLformula} are presented at the end of Section~\ref{sec:convergence}.
The remaining work is divided into two parts: In the following section, we translate the spectrum of $\Lambda(\sigma)$ by $\tau > 0$ to the right of the origin and prove the Fr\'echet differentiability of the resulting logarithmic isomorphism $\log\!\Lambda_\tau(\sigma): L^2_\diamond(\partial \Omega) \to L^2_\diamond(\partial \Omega)$ as well as some other useful properties. Section~\ref{sec:convergence} is then devoted to checking that everything stays intact when $\tau$ tends to zero. In addition, auxiliary results on equivalent norms for $H_\diamond^r(\partial \Omega)$, $r \in [-\tfrac{1}{2},\tfrac{1}{2}]$, defined using the $\sigma$-dependent singular system of $\Lambda(\sigma)$ are presented in Appendix~\ref{app:norms}.

\begin{remark}
	As the completely logarithmic forward map $L$ defined by \eqref{eq:log_for} is the one that was found to be closer to linear than, say, the standard forward map $\sigma \mapsto \Lambda(\sigma)$ in the numerical studies of \cite{Hyvonen18}, our main results presented above are formulated for $L$. However, when proving Theorem~\ref{thm:main}, it is more natural to first consider
	\begin{equation}
	\label{eq:log_for2}
	\log\!\Lambda: \sigma \mapsto \log\!\Lambda(\sigma)
	\end{equation}
	that maps $L^\infty_+(\Omega)$ to $\mathscr{L}(H^\epsilon_\diamond(\partial \Omega), H^{-\epsilon}_\diamond(\partial \Omega))$ and subsequently resort to the chain rule for Banach spaces.
\end{remark}

\section{Shifted Neumann-to-Dirichlet map $\Lambda_\tau$ and its logarithm}
\label{sec:shifted}
Adding a positive multiple $\tau > 0$ of identity to the ND map allows to define a shifted logarithmic ND map that is bounded on all of $L^2_\diamond(\partial\Omega)$ and can be differentiated with respect to the conductivity by means of standard functional calculus. This section provides uniform norm bounds with respect to $\tau$ for the first and second order derivatives of this shifted logarithmic boundary map in $\mathscr{L}(L^\infty(\Omega),\mathscr{L}(L^2_\diamond(\partial\Omega)))$ and $\mathscr{L}(L^\infty(\Omega)^2,\mathscr{L}(L^2_\diamond(\partial\Omega)))$, respectively. Such bounds guarantee the existence of well defined limit operators for these derivatives as $\tau\to 0^+$. In particular, the limit of the first derivative provides a natural candidate for the searched for $D\!\log\!\Lambda(\sigma; \, \cdot \,)$ in $\mathscr{L}(L^\infty(\Omega),\mathscr{L}(L^2_\diamond(\partial\Omega)))$.

\subsection{Definition and basic properties}
We define the shifted Neumann-to-Dirichlet map, or a Neumann-to-Robin map, as
\begin{equation}
\label{eq:shift_Lambda}
\Lambda_\tau(\sigma):= \Lambda(\sigma) + \tau I:  L^2_\diamond(\partial \Omega) \to L^2_\diamond(\partial \Omega),
\end{equation}
where $I : L^2_\diamond(\partial \Omega) \to L^2_\diamond(\partial \Omega)$ is the identity map.
Obviously,
\begin{equation*}
\Lambda_\tau(\sigma)f =   U + \tau f = \sum_{k=1}^\infty (\lambda_k + \tau) \langle f, \phi_k \rangle \, \phi_k,
\end{equation*}
where the mean-free Dirichlet boundary value $U \in H^{1/2}_\diamond(\partial \Omega)$ is defined by \eqref{eq:Udef} and the singular system $\{ \lambda_k, \phi_k\}_{k=1}^\infty \subset \R_+ \times H^{1/2}_\diamond(\partial \Omega)$ is as in \eqref{eq:spectral}. In particular, the spectrum of $\Lambda_\tau(\sigma)$ is contained in the interval $[\tau,\tau+\|\Lambda(\sigma)\|_{\mathscr{L}(L^2_\diamond(\partial\Omega))}]$. The following remark summarizes some obvious properties of $\Lambda_\tau(\sigma)$.

\begin{remark}
	\label{remark:N-to-R}
	For any $\sigma \in L^\infty_+(\Omega)$ and $\tau >0$, the shifted ND map $\Lambda_\tau(\sigma): L^2_\diamond(\partial \Omega) \to L^2_\diamond(\partial \Omega)$ is a self-adjoint isomorphism.
	The spectrum of $\Lambda_\tau(\sigma)$ consists solely of its strictly positive eigenvalues $\lambda_{k,\tau} := \lambda_k + \tau$, $k \in \N$, and their accumulation point $\tau$. Moreover, $\Lambda_\tau(\sigma)$ extends to an isomorphism from $H^{-1/2}_\diamond(\partial \Omega)$ to itself (denoted by the same symbol).
\end{remark}

Observe that the shifted forward map induced by \eqref{eq:shift_Lambda}, i.e., 
\begin{equation*}
L^\infty_+(\Omega) \ni \sigma \mapsto \Lambda_\tau(\sigma) \in \mathscr{L}(H^{-1/2}_\diamond(\partial \Omega)),
\end{equation*}
obviously has the same Fr\'echet derivative as the original unshifted version \eqref{eq:lambda}, since the perturbation $\tau I$ is independent of $\sigma$. In particular, the derivative $D\Lambda_\tau(\sigma; \eta) = D\Lambda(\sigma; \eta)$ belongs to $\mathscr{L}(H^{-1/2}_\diamond(\partial \Omega), H^{1/2}_\diamond(\partial \Omega))$  and is thus more smoothening than $\Lambda_\tau(\sigma)$ itself.

\subsection{Differentiability of \texorpdfstring{\boldmath$\log\!\Lambda_\tau$}{shifted log ND}}
We define the logarithm of $\Lambda_\tau(\sigma)$ in the natural manner, i.e.,
\begin{equation}
\label{eq:logspectral_tau}
\log\!\Lambda_\tau(\sigma) \colon f \mapsto \sum_{k=1}^\infty \log(\lambda_k + \tau) \langle f,\phi_k \rangle \, \phi_k, \qquad \tau > 0.
\end{equation}
It is obvious that $\log\!\Lambda_\tau(\sigma)$ is a bounded self-adjoint operator on $L^2_\diamond(\partial\Omega)$ for any $\tau > 0$
as its spectrum lies on the bounded interval $[\log(\tau), \log(\tau + \lambda_1)] \subset \R$ (but it is only self-adjoint for $\tau = 0$).
Moreover, $\log\!\Lambda_\tau(\sigma)-\log(\tau)I$ is compact since its eigenvalues $\log(\tfrac{\lambda_k}{\tau}+1)$ converge to zero as $k \to \infty$. This implies that $\log\!\Lambda_\tau(\sigma_2)-\log\!\Lambda_\tau(\sigma_1)$ is also compact for any $\sigma_1,\sigma_2\in L_+^\infty(\Omega)$ and $\tau>0$. 

As $\Lambda_\tau(\sigma)$ is itself an element of $\mathscr{L}(L^2_\diamond(\partial \Omega))$, the definition \eqref{eq:logspectral_tau} coincides with the more general Riesz--Dunford formula
\begin{equation}
\log(T) = \frac{1}{2\pi {\rm i}} \int_{\Gamma} \log(z) (z I - T)^{-1} {\rm d} z, \label{eq:riesz-dunford}
\end{equation}
that is valid for any $T \in \mathscr{L}(L^2_\diamond(\partial \Omega))$ for which there exists a positively oriented rectifiable Jordan curve $\Gamma \subset \C$ that encloses the spectrum of $T$ without intersecting the closed negative real axis. 

We also introduce the `shifted logarithmic' forward operator
\begin{equation}
\label{eq:log_for_tau}
\log\!\Lambda_\tau: \sigma \mapsto \log\!\Lambda_\tau(\sigma), \qquad L^\infty_+(\Omega) \to \mathscr{L}(L^2_\diamond(\partial \Omega))
\end{equation}
for any $\tau > 0$. Because the natural logarithm is analytic in a neighborhood of the spectrum of $\Lambda_\tau(\sigma)$, we have the following preliminary differentiability result. In what follows, we employ the shorthand notation $\Lambda^{-1}_{\tau+s}(\sigma) := (\Lambda_{\tau + s}(\sigma))^{-1}$; 
recall from Remark~\ref{remark:N-to-R} that $\Lambda^{-1}_{\tau+s}(\sigma) \in \mathscr{L}(L^2_\diamond(\partial \Omega))$ for $\tau + s>0$ because the spectrum of $\Lambda_{\tau + s}(\sigma)$ lies on the interval $[\tau + s, \lambda_1 + \tau + s]$ not containing the origin.

\begin{lemma}
	\label{lemma:help_diff}
	The operator logarithm
	is differentiable at $\Lambda_\tau(\sigma)$ in the sense that there exists $D \! \log(\Lambda_\tau(\sigma); \, \cdot \,) \in \mathscr{L}(\mathscr{L}_{\rm sa}(L^2_\diamond(\partial \Omega)))$ such that
	\begin{equation*}
	\frac{1}{\|S \|} \big \| \log (\Lambda_\tau(\sigma) + S) - \log\!\Lambda_\tau(\sigma)  - D \! \log( \Lambda_\tau(\sigma); S) \big \| \to 0 \qquad {\rm as} \ \ \|S\| \to 0
	\end{equation*}
	for a self-adjoint perturbation $S$ and with $\| \cdot \| = \| \cdot \|_{\mathscr{L}(L^2_\diamond(\partial \Omega))}$. 
The derivative allows the representation
	\begin{equation} \label{eq:operatorlog}
	D \!\log(\Lambda_\tau(\sigma); S) = \int_0^\infty \Lambda_{\tau+s}^{-1}(\sigma) \, S \, \Lambda_{\tau+s}^{-1}(\sigma) \, {\rm d} s
	\end{equation}
	understood in the sense of a Bochner integral. 
	
	Moreover, if $S$ can be extended to an operator in $\mathscr{L}(H_\diamond^{-1/2}(\partial\Omega),H_\diamond^{1/2}(\partial\Omega))$ and $0<\varsigma_-\leq\sigma\leq \varsigma_+<\infty$ almost everywhere in $\Omega$, then
	\begin{equation}
	\| D \!\log(\Lambda_\tau(\sigma); S)\|_{\mathscr{L}(L^2_\diamond(\partial\Omega))} \leq C\| S \|_{\mathscr{L}(H_\diamond^{-1/2}(\partial\Omega),H_\diamond^{1/2}(\partial\Omega))}, \label{eq:Sbnd}
	\end{equation}
	where $C = C(\Omega,\varsigma_-,\varsigma_+)$ is independent of $\sigma$ and $\tau > 0$.
\end{lemma}

\begin{proof}
  The differentiability result and the representation \eqref{eq:operatorlog} follow from the material in~\cite{Pedersen00};
in particular, see the last formula of Section~4.3 on page 155 in \cite{Pedersen00}.

What remains to be proven is \eqref{eq:Sbnd}. 
By \cite[Theorem 3.7.3, p.~78]{Hille_1957}\footnote{Observe that the Pettis integral is a generalization of the Bochner integral.} 
continuous linear maps commute with Bochner integrals. Hence, we may utilize the self-adjointness of  $D \!\log(\Lambda_\tau(\sigma); S)$ to estimate
	\begin{align}
\label{eq:logdernorm}
	\| D \!\log(\Lambda_\tau(\sigma); S)\|_{\mathscr{L}(L^2_\diamond(\partial\Omega))} \hspace{-2cm}& \nonumber \\[1mm]
	&=
	\sup_{\|f \|_{L^2(\partial\Omega)}=1} \big|\langle f, \, D \!\log(\Lambda_\tau(\sigma); S)f \rangle \big| \nonumber  \\
	&\leq \sup_{\|f \|_{L^2(\partial\Omega)}=1} \int_0^\infty \big| \big\langle f,  \, \Lambda_{\tau+s}^{-1}(\sigma)  S \Lambda_{\tau+s}^{-1}(\sigma) f\big \rangle \big| \, {\rm d} s  \nonumber\\
	&= \sup_{\|f \|_{L^2(\partial\Omega)}=1} \int_0^\infty \big| \big\langle \Lambda_{\tau+s}^{-1}(\sigma)  f,   \, S  \Lambda_{\tau+s}^{-1}(\sigma) f \big \rangle \big| \, {\rm d} s  \nonumber \\
	&\leq \| S \|_{\mathscr{L}(H_\diamond^{-1/2}(\partial\Omega),H_\diamond^{1/2}(\partial\Omega))}\sup_{\|f \|_{L^2(\partial\Omega)}=1} \int_0^\infty \| \Lambda_{\tau+s}^{-1}(\sigma)  f \|_{H^{-1/2}(\partial\Omega)}^2 \, {\rm d} s,
	\end{align}
	where we employed the boundedness of the dual bracket over $H^{-1/2}_\diamond(\partial\Omega) \times H^{1/2}_\diamond(\partial\Omega)$.  

	We now resort to the equivalent norm $\| \cdot  \|_{-1/2,\sigma}$ for $H^{-1/2}_\diamond(\partial \Omega)$ defined by \eqref{eq:s_sigma_norm}, for which it obviously holds $\| \Lambda_{\tau+s}^{-1}(\sigma)  f \|_{-1/2,\sigma}\leq \| \Lambda_{s}^{-1}(\sigma)  f \|_{-1/2,\sigma}$ for any $\tau, s > 0$ and $f \in L^2_{\diamond}(\partial \Omega)$. In particular, we have
	\begin{equation}
	\|\Lambda_{\tau+s}^{-1}(\sigma) f\|_{H^{-1/2}(\partial\Omega)}^2 \leq C \|\Lambda_{s}^{-1}(\sigma) f\|_{-1/2,\sigma}^2 = C \sum_{k=1}^\infty \frac{\lambda_k}{(s+\lambda_k)^2} |\langle f , \phi_k \rangle|^2, \label{eq:laminvts}
	\end{equation}
where $C = C(\Omega, \varsigma_-, \varsigma_+)$ can be chosen to be independent of $\sigma$ itself due to Theorem~\ref{thm:norm_equi}. Denoting $\|\cdot\| = \|\cdot\|_{\mathscr{L}(H_\diamond^{-1/2}(\partial\Omega),H_\diamond^{1/2}(\partial\Omega))}$ and plugging \eqref{eq:laminvts} into \eqref{eq:logdernorm}, we finally obtain
	\begin{equation*}
	\| D \!\log(\Lambda_\tau(\sigma); S)\|_{\mathscr{L}(L^2_\diamond(\partial\Omega))} \leq C \| S \| \sup_{\|f\|_{L^2(\partial\Omega)}=1} \sum_{k=1}^\infty |\langle f , \phi_k \rangle|^2 \int_0^\infty \frac{\lambda_k}{(s+\lambda_k)^2}\, {\rm d} s = C \| S \|,
	\end{equation*} 
	since $\{\phi_k\}_{k\in\mathbb{N}}$ is an orthonormal basis for $L_\diamond^2(\partial\Omega)$.
\end{proof}

Due to \eqref{eq:Sbnd}, the right-hand side of \eqref{eq:operatorlog} defines a self-adjoint operator in $\mathscr{L}(L^2_\diamond(\partial \Omega))$ for any $S\in\mathscr{L}(H_\diamond^{-1/2}(\partial\Omega),H_\diamond^{1/2}(\partial\Omega)) \cap \mathscr{L}_{\rm sa}(L^2_\diamond(\partial \Omega))$ even if $\tau = 0$. This limit behavior will be central in the following.

In Lemma~\ref{lemma:help_diff} the perturbation $S$ is an arbitrary element of $\mathscr{L}_{\rm sa}(L^2_\diamond(\partial \Omega))$. However, we are actually only interested in perturbations of $\Lambda_\tau(\sigma)$ induced by a change in the conductivity $\sigma$. This observation leads to the following proposition.

\begin{proposition}
	\label{prop:logtau_diff}
	For $\tau\geq 0$ and $\sigma\in L^\infty_+(\Omega)$, we define an operator $DF_\tau(\sigma; \, \cdot \, )$ belonging to $\mathscr{L}(L^\infty(\Omega), \mathscr{L}(L^2_\diamond(\partial \Omega)))$ via
	\begin{equation}
	\label{eq:the_derivative}
	DF_\tau(\sigma; \eta) := \int_0^\infty \Lambda_{\tau+s}^{-1}(\sigma) \, D\Lambda(\sigma; \eta) \, \Lambda_{\tau+s}^{-1}(\sigma)\, {\rm d} s.
	\end{equation}
	The following hold:
	\begin{enumerate}[(i)]
		\item If $0 < \varsigma_- \leq \sigma \leq \varsigma_+ < \infty$ almost everywhere in $\Omega$, then
		\begin{equation}
		\label{eq:deriv_norm}
		\| DF_\tau(\sigma; \, \cdot \, ) \|_{\mathscr{L}(L^\infty(\Omega), \mathscr{L}(L^2_\diamond(\partial \Omega)))} \leq
		C,
		\end{equation}
		where $C = C(\Omega, \varsigma_-, \varsigma_+)$ is independent of $\sigma$ and $\tau \geq 0$.
		\item If $\tau>0$, the shifted logarithmic forward operator $\log\!\Lambda_\tau$ defined by \eqref{eq:log_for_tau} is continuously Fr\'echet differentiable with the derivative
		\begin{equation*}
		D\!\log\!\Lambda_{\tau}(\sigma; \, \cdot \,) = DF_\tau(\sigma;\, \cdot \,).
		\end{equation*}
	\end{enumerate}
\end{proposition}

\begin{proof}
	For $\log\!\Lambda_{\tau}$ with $\tau>0$, the continuous Fr\'echet differentiability and the representation \eqref{eq:the_derivative} for its derivative follow from the combination of Lemma~\ref{lemma:help_diff}, the continuous Fr\'echet differentiability of the standard forward map $L^\infty_+(\Omega) \ni \sigma \mapsto \Lambda(\sigma) \in \mathscr{L}(H^{-1/2}_\diamond(\partial \Omega), H^{1/2}_\diamond(\partial \Omega)) \subset \mathscr{L}(L^2_\diamond(\partial \Omega))$ considered in Appendix~\ref{app:derivatives}, the chain rule for Banach spaces, and the continuous Fr\'echet differentiability of the operator logarithm. Indeed, the derivative of the operator logarithm is continuous in some neighborhood of any strictly positive definite element of $\mathscr{L}_{\rm sa}(L^2_\diamond(\partial \Omega))$, which applies in particular to $\Lambda_\tau(\sigma)$ for any $\sigma \in L^\infty_+(\Omega)$ and $\tau>0$ (cf.~\cite{Pedersen00}). See also the comment succeeding Remark~\ref{remark:N-to-R}.
	
	To prove the bound \eqref{eq:deriv_norm}, we first of all note that  $D\Lambda(\sigma; \eta): L^2_\diamond(\partial \Omega) \to L^2_\diamond(\partial \Omega)$ is self-adjoint for any $\sigma \in L^\infty_+(\Omega)$ and $\eta \in L^\infty(\Omega)$; see Appendix~\ref{app:derivatives}. By virtue of \eqref{eq:Sbnd}, the comment after the proof of Lemma~\ref{lemma:help_diff}, and \eqref{eq:FrDB}, we thus get
	\begin{equation*}
	\| DF_\tau(\sigma; \eta)\|_{\mathscr{L}(L^2_\diamond(\partial \Omega))} \leq C(\Omega, \varsigma_-, \varsigma_+)  \|\eta \|_{L^\infty(\Omega)},
	\end{equation*}
	for any $\tau \geq 0$. The claim then follows by taking the supremum over $\eta \in L^\infty(\Omega)$ satisfying $\| \eta \|_{L^\infty(\Omega)} = 1$.
\end{proof}

According to Proposition~\ref{prop:logtau_diff}, $DF_\tau(\sigma; \eta): L^2_\diamond(\partial \Omega) \to L^2_\diamond(\partial \Omega)$ is well defined and bounded even for $\tau=0$. In particular, $DF_0(\sigma; \, \cdot \,)$ provides the natural candidate for the Fr\'echet derivative of the (unshifted) logarithmic forward map from \eqref{eq:log_for2}. 

We next introduce an alternative representation for $DF_\tau(\sigma; \eta)$, arguably more suitable for numerical considerations. Take note that this formula is known to hold for $\tau > 0$ due to,~e.g.,~\cite[Corollary~2.3]{Gilliam_2009} and employing the Riesz--Dunford formula \eqref{eq:riesz-dunford} to define $\log\!\Lambda_{\tau}(\sigma)$. However, we also consider the case $\tau=0$ here. The following proposition can be considered a prequel for Corollary~\ref{cor:DLformula}.
\begin{proposition} \label{prop:DLambdaformula}
	Let $\tau\geq 0$, $\sigma\in L^\infty_+(\Omega)$, and $\eta\in L^\infty(\Omega)$. The following representation in the eigenbasis $\{\phi_k\}_{k\in\mathbb{N}}$ is valid for any $f\in L^2_\diamond(\partial\Omega)$:
	\begin{equation*} 
	DF_\tau(\sigma; \eta)f = \sum_{j=1}^\infty \sum_{k=1}^\infty c_{j,k,\tau}\, \langle f,\phi_k\rangle \big\langle D\Lambda(\sigma,\eta)\phi_k,\phi_j\big\rangle \, \phi_j,
	\end{equation*}
	where 
	\begin{equation*}
	c_{j,k,\tau} := \begin{cases}
	\frac{\log(\lambda_j+\tau)-\log(\lambda_k+\tau)}{\lambda_j-\lambda_k}, &\quad \lambda_j\neq \lambda_k, \\[1mm]
	\frac{1}{\lambda_j+\tau}, & \quad \lambda_j=\lambda_k.
	\end{cases}
	\end{equation*}	
\end{proposition}
\begin{proof}
	Since $DF_\tau(\sigma; \eta) \in \mathscr{L}(L^2_\diamond(\partial\Omega))$ for $\tau\geq 0$ by Proposition~\ref{prop:logtau_diff}, we are allowed to write
	\begin{equation*}
	DF_\tau(\sigma; \eta)f = \sum_{j=1}^\infty \sum_{k=1}^\infty \langle f,\phi_k\rangle \big\langle DF_\tau(\sigma; \eta)\phi_k,\phi_j \big\rangle \phi_j.
	\end{equation*}	
	The proof is concluded by expanding $\langle DF_\tau(\sigma; \eta)\phi_k,\phi_j\rangle$. To this end, we employ \eqref{eq:the_derivative} and the spectral decomposition of $\Lambda_{\tau+s}^{-1}(\sigma)$ to deduce
	\begin{align*}
	\langle DF_\tau(\sigma; \eta)\phi_k,\phi_j\rangle \hspace{-2cm}&\\[1mm]
	&= \int_0^\infty \big\langle D\Lambda(\sigma;\eta)\Lambda_{\tau+s}^{-1}(\sigma)\phi_k,\Lambda_{\tau+s}^{-1}(\sigma)\phi_j \big\rangle \, {\rm d} s \\
	&= \int_0^\infty \sum_{m=1}^\infty\sum_{n=1}^\infty \frac{1}{(\lambda_m+\tau+s)(\lambda_n+\tau+s)}\langle \phi_k,\phi_m\rangle \big\langle D\Lambda(\sigma;\eta)\phi_m,\phi_n \big \rangle \overline{\langle \phi_j,\phi_n\rangle} \, {\rm d} s \\
	&= \,\langle D\Lambda(\sigma;\eta)\phi_k,\phi_j \rangle \int_0^\infty \frac{1}{(\lambda_k+\tau+s)(\lambda_j+\tau+s)} \, {\rm d} s,
	\end{align*}
	where the integral over $s$ gives $c_{j,k,\tau}$.
\end{proof}

To complete this section, we still need to consider the second Fr\'echet derivative of $\log\!\Lambda_{\tau}$ as well as its uniform boundedness with respect to $\tau>0$.

\begin{proposition}
	\label{prop:logtau_2diff}
	For $\tau\geq 0$ and $\sigma\in L^\infty_+(\Omega)$, we define an operator $D^2F_\tau(\sigma; \, \cdot \,, \, \cdot \, )$ belonging to $\mathscr{L}(L^\infty(\Omega)^2,\mathscr{L}(L^2_\diamond(\partial \Omega)))$ via
	\begin{align}
	\label{eq:second_derivative}
	D^2F_\tau(\sigma; \eta, \xi)  &:= \int_0^\infty \Lambda_{\tau+s}^{-1}(\sigma) \, D^2\Lambda(\sigma; \eta, \xi) \, \Lambda_{\tau+s}^{-1}(\sigma) \, {\rm d} s \\
	& \phantom{{}:={}} - \int_0^\infty \Lambda_{\tau+s}^{-1}(\sigma) \,  D\Lambda(\sigma; \eta) \, \Lambda_{\tau+s}^{-1}(\sigma) \, D\Lambda(\sigma; \xi) \, \Lambda_{\tau+s}^{-1}(\sigma) \, {\rm d} s \nonumber \\
	& \phantom{{}:={}} - \int_0^\infty \Lambda_{\tau+s}^{-1}(\sigma) \,  D\Lambda(\sigma; \xi) \, \Lambda_{\tau+s}^{-1}(\sigma) \, D\Lambda(\sigma; \eta) \, \Lambda_{\tau+s}^{-1}(\sigma) \, {\rm d} s. \nonumber
	\end{align}
	The following hold:
	\begin{enumerate}[(i)]
		\item If $0 < \varsigma_- \leq \sigma \leq \varsigma_+ < \infty$ almost everywhere in $\Omega$, then
		\begin{equation}
		\label{eq:2deriv_norm}
		\| D^2F_\tau(\sigma; \, \cdot \, , \, \cdot \,) \|_{\mathscr{L}(L^\infty(\Omega)^2,\mathscr{L}(L^2_\diamond(\partial \Omega)))} \leq C ,
		\end{equation}
		where $C = C(\Omega, \varsigma_-, \varsigma_+)$ is independent of $\sigma$ and $\tau \geq 0$.
		\item If $\tau>0$, the shifted logarithmic forward operator $\log\!\Lambda_\tau$ defined by \eqref{eq:log_for_tau} is twice Fr\'echet differentiable with the second order derivative
		\begin{equation*}
		D^2\!\log\!\Lambda_{\tau}(\sigma; \, \cdot \, , \, \cdot \,) = D^2F_\tau(\sigma;\, \cdot \, , \, \cdot \,).
		\end{equation*}
	\end{enumerate}
\end{proposition}
\begin{proof}
	For $\log\!\Lambda_{\tau}$ with $\tau > 0$, the representation \eqref{eq:second_derivative} for the second derivative is a consequence of Proposition~\ref{prop:logtau_diff}, the product rule for Banach spaces, the differentiation formula 
	\begin{equation*}
	D\Lambda_{\tau+s}^{-1}(\sigma; \, \cdot \, ) = -\Lambda_{\tau+s}^{-1}(\sigma)D\Lambda_{\tau+s}(\sigma;\, \cdot \, )\Lambda_{\tau+s}^{-1}(\sigma) = -\Lambda_{\tau+s}^{-1}(\sigma)D\Lambda(\sigma;\, \cdot \, )\Lambda_{\tau+s}^{-1}(\sigma),
	\end{equation*}
	and Hille's theorem (see \cite[Lemma~1]{Hille_1952} or \cite[Theorem 3.7.12, p.~83]{Hille_1957}) that requires the Bochner integrals resulting from the differentiation under the integral sign are well defined. This condition is satisfied by virtue of \eqref{eq:2deriv_norm} that is established below.
	
	The bound \eqref{eq:2deriv_norm} follows from the triangle inequality after separately estimating the three terms $I_1$, $I_2$, and $I_3$ (defined in the order they appear) on the right-hand side of \eqref{eq:second_derivative}. The first one can be handled by resorting to \eqref{eq:Sbnd} and \eqref{eq:FrDB2}:
	\begin{equation*}
	\|I_1\|_{\mathscr{L}(L^2_\diamond(\partial \Omega))} \leq C\|\eta\|_{L^\infty(\Omega)}\|\xi\|_{L^\infty(\Omega)},
	\end{equation*}	
where $C = C(\Omega, \varsigma_-,\varsigma_+)> 0$ can be chosen to be independent of $\sigma$.
	Combining \eqref{eq:Sbnd} next with \eqref{eq:FrDB}, it is easy to see that the other two terms satisfy
	\begin{equation*}
	\|I_j\|_{\mathscr{L}(L^2_\diamond(\partial \Omega))} \leq C\|\eta\|_{L^\infty(\Omega)} \|\xi\|_{L^\infty(\Omega)} \| \Lambda_{\tau+s}^{-1}(\sigma)\|_{\mathscr{L}(H_\diamond^{1/2}(\partial\Omega),H^{-1/2}_\diamond(\partial\Omega))}, \qquad j=2,3,
	\end{equation*}
where  $C = C(\Omega, \varsigma_-,\varsigma_+)> 0$ can once again be chosen independently of the actual $\sigma$.
	To estimate $ \| \Lambda_{\tau+s}^{-1}(\sigma) \|_{\mathscr{L}(H_\diamond^{1/2}(\partial\Omega),H^{-1/2}_\diamond(\partial\Omega))}$, we first use the equivalent norm $\| \cdot \|_{1/2, \sigma}$ for $H^{1/2}_\diamond(\partial \Omega)$ defined by \eqref{eq:s_sigma_norm} along with Theorem~\ref{thm:norm_equi} to deduce
	\begin{align*}
	\| \Lambda_{\tau+s}^{-1}(\sigma) \|_{\mathscr{L}(H_\diamond^{1/2}(\partial\Omega),H^{-1/2}_\diamond(\partial\Omega))} &\leq C(\Omega, \varsigma_-,\varsigma_+) \| \Lambda^{-1}(\sigma) \|_{\mathscr{L}(H_\diamond^{1/2}(\partial\Omega),H^{-1/2}_\diamond(\partial\Omega))} \\
	&\leq C(\Omega, \varsigma_-,\varsigma_+)\| \Lambda^{-1}(\varsigma_+) \|_{\mathscr{L}(H_\diamond^{1/2}(\partial\Omega),H^{-1/2}_\diamond(\partial\Omega))} \\
&\leq C(\Omega, \varsigma_-,\varsigma_+),
	\end{align*}
where the second step is a consequence of the monotonicity relation in Lemma~\ref{lemma:inverse} for $\sigma\leq \varsigma_+$ together with $\Lambda^{-1}(\sigma)$ and $\Lambda^{-1}(\varsigma_+)$ being symmetric. 
	Combining the preceding estimates finally gives
	\begin{equation*}
	\|D^2F_\tau(\sigma; \eta, \xi)\|_{\mathscr{L}(L^2_\diamond(\partial \Omega))} \leq C(\Omega, \varsigma_-,\varsigma_+) \|\eta\|_{L^\infty(\Omega)} \|\xi\|_{L^\infty(\Omega)},
	\end{equation*}
	and the proof is completed by taking the supremum over $\eta, \xi \in L^\infty(\Omega)$ such that $\| \eta \|_{L^\infty(\Omega)} = \|\xi \|_{L^\infty(\Omega)} = 1$. 
\end{proof}

\begin{remark}
  The continuity of the second derivative $\sigma \mapsto D^2\!\log\!\Lambda_{\tau}(\sigma; \, \cdot \, , \, \cdot \,)$ for $\tau > 0$ could be established,~e.g.,~by repeating the arguments in the proof of Proposition~\ref{prop:logtau_2diff} to show that $\log\!\Lambda_{\tau}$, $\tau > 0$, is actually three times Fr\'echet differentiable. However, we skip such technical calculations because the continuity of the second derivative is not actually needed when proving the main results of this work.
  \end{remark}

\section{Proofs of the main results}
\label{sec:convergence}

We now verify that $D\!\log\!\Lambda:= DF_0$, defined by \eqref{eq:the_derivative}, is in fact the Fr\'echet derivative of $\log\!\Lambda$, and for this reason we call $DF_0$ by its new name in the rest of this section. We start  by considering the convergence of $\log\!\Lambda_\tau(\sigma)$ and $D\!\log\!\Lambda_\tau(\sigma; \, \cdot \,)$ towards $\log\!\Lambda(\sigma)$ and $D\!\log\!\Lambda(\sigma; \, \cdot \,)$, respectively, as $\tau>0$ tends to zero.

\begin{proposition}
	\label{prop:tau_to-zero}
	If $0 < \varsigma_- \leq \sigma \leq \varsigma_+ < \infty$ almost everywhere in $\Omega$, then for any fixed $0 < \epsilon \leq \frac{1}{2}$, 
	\begin{equation}
	\label{eq:tau_to-zero1}
	\| \log\!\Lambda(\sigma) - \log\!\Lambda_\tau(\sigma) \|_{\mathscr{L}(H^{\epsilon}_\diamond(\partial \Omega), H^{-\epsilon}_{\diamond}(\partial\Omega))} \leq C\epsilon^{-1} \tau^{2\epsilon} 
	\end{equation}
	and
	\begin{equation}
	\label{eq:tau_to-zero2}
	\| D\!\log\!\Lambda(\sigma; \, \cdot \, ) - D\!\log\!\Lambda_\tau(\sigma; \, \cdot \, ) \|_{\mathscr{L}(L^\infty(\Omega), \mathscr{L}(H^{\epsilon}_\diamond(\partial \Omega), H^{-\epsilon}_{\diamond}(\partial\Omega)))} \leq C \tau^\epsilon,
	\end{equation}
	where $C = C(\Omega,\varsigma_-,\varsigma_+) >0$ does not depend on $\sigma$, $\tau \geq 0$, or $\epsilon$. 
\end{proposition}
\begin{proof}
	Using the spectral decompositions \eqref{eq:logspectral} and \eqref{eq:logspectral_tau}, we get
	\begin{equation*}
	\big(\log\!\Lambda(\sigma) - \log\!\Lambda_\tau(\sigma)\big) f
	= - \sum_{k=1}^\infty  \log \! \Big(1 + \frac{\tau}{\lambda_k} \Big) \langle f, \phi_k \rangle \phi_k 
	\end{equation*}
	for any $f\in H^{\epsilon}_\diamond(\partial \Omega)$. Hence,
	\begin{align}
	\label{eq:lim1}
	\big\| (\log\!\Lambda(\sigma) - \log\!\Lambda_\tau(\sigma)) f \big\|_{H^{-\epsilon}(\partial\Omega)}^2
	&\leq  C\sum_{k=1}^\infty  \lambda_k^{2\epsilon} \log^2 \!  \Big(1 + \frac{\tau}{\lambda_k} \Big) | \langle f, \phi_k \rangle|^2 \nonumber \\[1mm]
	&\leq C\sup_{t \in \R_+} t^{4\epsilon} \log^2\!\Big(1 + \frac{\tau}{t} \Big)
	\sum_{k=1}^\infty \lambda_k^{-2\epsilon} | \langle f, \phi_k \rangle|^2 \nonumber \\[1mm]
	&\leq C \sup_{t \in \R_+} t^{4\epsilon} \log^2\!\Big(1 + \frac{\tau}{t} \Big) \| f\|_{H^{\epsilon}(\partial \Omega)}^2,
	\end{align}
	where we used Theorem~\ref{thm:norm_equi} that also indicates $C = C(\Omega, \varsigma_-,\varsigma_+)$ is independent of $\sigma$ and $f$. Since $\epsilon\in(0,\tfrac{1}{2}]$, we have 
$$
\log(1 + y) \leq \frac{1}{2 \epsilon} y^{2\epsilon} \qquad  {\rm for} \ {\rm all} \ y\geq 0. 
$$
This estimate can be easily proven,~e.g.,~by first observing that it holds at $y = 0$ and then comparing the derivatives of the two sides.
Hence, taking the square root and supremum over $f \in H^{\epsilon}_\diamond(\partial \Omega)$ with $\| f \|_{H^{\epsilon}(\partial \Omega)}=1$ in \eqref{eq:lim1}, one arrives at
	\begin{equation*}
	\| \log\!\Lambda(\sigma) - \log\!\Lambda_\tau(\sigma)\|_{\mathscr{L}(H^{\epsilon}_\diamond(\partial \Omega), H^{-\epsilon}_{\diamond}(\partial\Omega))} \leq C\epsilon^{-1}\tau^{2\epsilon},
	\end{equation*}
which proves the first part of the claim.
	
In order to prove \eqref{eq:tau_to-zero2}, we first consider the difference of $\Lambda^{-1}_{\tau+s}(\sigma)$ and $\Lambda^{-1}_{s}(\sigma)$ for fixed $\tau, s>0$. To this end, we write
	\begin{equation*}
	(\Lambda^{-1}_{s}(\sigma) - \Lambda^{-1}_{\tau + s}(\sigma))f = \sum_{k=1}^\infty \frac{\tau}{(s+\lambda_k)(\tau+s+\lambda_k)}\langle f,\phi_k\rangle \phi_k
	\end{equation*}
	for $f\in L^2_\diamond(\partial\Omega)$. Hence, due to Theorem~\ref{thm:norm_equi},
	\begin{align}
	\big\|(\Lambda^{-1}_{s}(\sigma) - \Lambda^{-1}_{\tau + s}(\sigma))f \big\|_{H^{-1/2}(\partial\Omega)}^2 &\leq C\sum_{k=1}^\infty \frac{\lambda_k\tau^2}{(s+\lambda_k)^2(\tau+s+\lambda_k)^2}|\langle f,\phi_k\rangle|^2 \nonumber\\ 
	&\leq C\sup_{t\in\mathbb{R}_+}\omega_{\tau,s,\epsilon}^2(t)\sum_{k=1}^\infty \frac{\lambda_k}{(s+\lambda_k)^2}\lambda_k^{-2\epsilon}|\langle f,\phi_k\rangle|^2,  \label{eq:lamdiff}
	\end{align}
	where $C = C(\Omega,\varsigma_-,\varsigma_+)>0$ is independent of $\sigma$ and 
	\begin{equation*}
	\omega_{\tau,s,\epsilon}(t) := \frac{t^\epsilon\tau}{\tau+s+t}.
	\end{equation*}
	As $\omega_{\tau,s,\epsilon}: \R_+ \to \R_+$ vanishes both at the origin and at infinity, its maximum is found at
	\begin{equation*}
	t_* =   \frac{\epsilon}{1-\epsilon} (\tau + s),
	\end{equation*}
	where $\omega'_{\tau,s,\epsilon}$ is zero. Hence,
	\begin{equation}
	\label{eq:rho2}
	\sup_{t \in \R_+} \omega_{\tau,s,\epsilon}(t) =  \epsilon^\epsilon(1-\epsilon)^{1-\epsilon} \frac{\tau}{(\tau + s)^{1-\epsilon}} \leq  \tau^{\epsilon},
	\end{equation}
	for all $\tau, s>0$ and $\epsilon\in(0,\tfrac{1}{2}]$.
	
	Note that $\sup_{k}\lambda_k = \|\Lambda(\sigma)\|_{\mathscr{L}(L^2_\diamond(\partial\Omega))} \leq \|\Lambda(\varsigma_-)\|_{\mathscr{L}(L^2_\diamond(\partial\Omega))}$ where the inequality is a consequence of monotonicity, cf.\ Appendix~\ref{app:norms}. In particular, 
	\begin{equation*}
	\lambda_k^{2\epsilon} \leq \max \big\{1,\|\Lambda(\varsigma_-)\|_{\mathscr{L}(L^2_\diamond(\partial\Omega))} \big\}, \qquad k \in \N.
	\end{equation*}
	Hence, we have
	\begin{equation}
	\big\|\Lambda_{s}^{-1}(\sigma) f \big\|_{H^{-1/2}(\partial\Omega)}^2 \leq C \sum_{k=1}^\infty \frac{\lambda_k}{(s+\lambda_k)^2}\lambda_k^{-2\epsilon} |\langle f , \phi_k \rangle|^2, \label{eq:laminvts2}
	\end{equation}
where $C = C(\Omega, \varsigma_-, \varsigma_+)$ is independent of $\sigma$ by Theorem~\ref{thm:norm_equi}.
	Mimicking the proof of Lemma~\ref{lemma:help_diff}, one may write
	\begin{align*}
	\big\langle f, \, (D\!\log\!\Lambda(\sigma; \eta)- D\!\log\!\Lambda_\tau(\sigma; \eta)) f \big \rangle \hspace{-4cm}& \\
	&=
	\int_0^\infty \Big( \big\langle \Lambda_{s}^{-1}(\sigma)  f,   \, D\Lambda(\sigma; \eta)  \Lambda_{s}^{-1}(\sigma) f \big \rangle  -  \big\langle \Lambda_{\tau + s}^{-1}(\sigma)  f,   \, D\Lambda(\sigma; \eta)  \Lambda_{\tau + s}^{-1}(\sigma) f \big \rangle \Big) \, {\rm d} s \\
	&=  \int_0^\infty  \big\langle \Lambda_{s}^{-1}(\sigma)  f,   \, D\Lambda(\sigma; \eta)(\Lambda_{s}^{-1}(\sigma) - \Lambda_{\tau + s}^{-1}(\sigma) ) f \big \rangle \, {\rm d} s \\
	&\phantom{{}={}}  + \int_0^\infty \big\langle (\Lambda_{s}^{-1}(\sigma) - \Lambda_{\tau + s}^{-1}(\sigma)) f,   \, D\Lambda(\sigma; \eta)  \Lambda_{\tau + s}^{-1}(\sigma) f \big \rangle  \, {\rm d} s.
	\end{align*}
	Because $D\!\log\!\Lambda(\sigma; \eta)- D\!\log\!\Lambda_\tau(\sigma; \eta): H^{\epsilon}_\diamond(\partial \Omega) \to H^{-\epsilon}_\diamond(\partial \Omega)$ is symmetric and by denoting $\|\cdot\| = \|\cdot\|_{H^{-1/2}(\partial\Omega)}$, we obtain
	\begin{align*}
	\| D\!\log\!\Lambda(\sigma; \eta) - D\!\log\!\Lambda_\tau(\sigma; \eta) \|_{\mathscr{L}(H^{\epsilon}_\diamond(\partial\Omega), H^{-\epsilon}_\diamond(\partial\Omega))}  \hspace{-4cm}&\\[2mm]
	&= \sup_{\|f\|_{H^{\epsilon}(\partial\Omega)}=1}  \big| \big\langle f, \, (D\!\log\!\Lambda(\sigma; \eta)- D\!\log\!\Lambda_\tau(\sigma; \eta)) f \big \rangle \big| \nonumber \\
	&\leq C \| \eta \|_{L^\infty(\Omega)} \sup_{\|f\|_{H^{\epsilon}(\partial\Omega)}=1} \Big( \int_0^\infty \| \Lambda_{s}^{-1}(\sigma)f \| \| \Lambda_{s}^{-1}(\sigma) - \Lambda_{\tau + s}^{-1}(\sigma) f\| \,  {\rm d} s \nonumber \\
	& \phantom{\leq C \| \eta \|_{L^\infty(\Omega)} \sup_{\|f\|_{H^{\epsilon}(\partial\Omega)}=1}} \  +  \int_0^\infty  \| \Lambda_{s}^{-1}(\sigma) - \Lambda_{\tau + s}^{-1}(\sigma) f\| \| \Lambda_{\tau + s}^{-1}(\sigma)f \| \,  {\rm d} s \Big) \nonumber \\
	&\leq C \| \eta \|_{L^\infty(\Omega)}  \sup_{\|f\|_{H^{\epsilon}(\partial\Omega)}=1} \int_0^\infty \| \Lambda_{s}^{-1}(\sigma) - \Lambda_{\tau + s}^{-1}(\sigma) f\| \| \Lambda_{s}^{-1}(\sigma)f \| \,  {\rm d} s, \nonumber
	\end{align*}
	where we used the triangle inequality, the boundedness of the dual bracket over $H^{-1/2}_\diamond(\partial \Omega) \times H^{1/2}_\diamond(\partial \Omega)$, and the estimates \eqref{eq:FrDB}, \eqref{eq:laminvts} and \eqref{eq:norm_equi}. Note that the generic constant $C = C(\Omega, \varsigma_-, \varsigma_+)$ remains independent of $\sigma$. 
	
	Substituting \eqref{eq:lamdiff}--\eqref{eq:laminvts2}, we finally get
	\begin{align*}
	\| D\!\log\!\Lambda(\sigma; \eta) - D\!\log\!\Lambda_\tau(\sigma; \eta) \|_{\mathscr{L}(H^{\epsilon}_\diamond(\partial\Omega), H^{-\epsilon}_\diamond(\partial\Omega))} \hspace{-6cm}&\\
	&\leq  C(\Omega, \varsigma_-,\varsigma_+)\tau^\epsilon  \|\eta \|_{L^\infty(\Omega)} \sup_{\|f\|_{H^{\epsilon}(\partial\Omega)}=1}  \sum_{k=1}^\infty  \lambda_k^{-2\epsilon} |\langle f, \phi_k \rangle |^2 \int_0^\infty \frac{\lambda_k}{(s + \lambda_k)^2} \, {\rm d} s \\
	& \leq C(\Omega, \varsigma_-,\varsigma_+)\tau^\epsilon  \|\eta \|_{L^\infty(\Omega)},
	\end{align*}
	where the last step is a simple consequence of Theorem~\ref{thm:norm_equi}. Taking the supremum over $\eta \in L^\infty(\Omega)$ with $\| \eta \|_{L^\infty(\Omega)} = 1$ proves \eqref{eq:tau_to-zero2}.
\end{proof}

Now we have developed all the necessary weaponry to prove our main results.

{\em Proof of Theorem~\ref{thm:main}}.
We prove the first part of Theorem~\ref{thm:main} for $\log\!\Lambda: \sigma \mapsto \log\!\Lambda(\sigma)$ in place of $L: \kappa \mapsto \log\!\Lambda({\rm e}^\kappa)$. Because $L = \log\!\Lambda \circ \exp$, the actual differentiability result for $L$ then immediately follows from the chain rule for Banach spaces. In the rest of this proof, $\|  \cdot  \| = \| \cdot  \|_{\mathscr{L}(H^{\epsilon}_\diamond(\partial \Omega), H^{-\epsilon}_{\diamond}(\partial\Omega))}$ for some fixed $\epsilon\in(0,\frac{1}{2}]$, if not explicitly stated otherwise.

To begin with, the triangle inequality yields
\begin{align}
\label{eq:three_eps}
\| \log\!\Lambda(\sigma + \eta) - \log\!\Lambda(\sigma) -  D\!\log\!\Lambda(\sigma; \eta) \| \hspace{-5cm}&  \nonumber \\[1mm]
& \leq \| \log\!\Lambda_\tau(\sigma + \eta) - \log\!\Lambda_\tau(\sigma) -  D\!\log\!\Lambda_\tau(\sigma; \eta) \| + \| D\!\log\!\Lambda_\tau(\sigma; \eta) - D\!\log\!\Lambda(\sigma; \eta) \| \nonumber \\[1mm]
& \phantom{{}\leq{}} + \| \log\!\Lambda(\sigma + \eta) - \log\!\Lambda_\tau(\sigma + \eta) \| + \| \log\!\Lambda_\tau(\sigma) - \log\!\Lambda(\sigma) \| \nonumber \\[1mm]
& \leq \| \log\!\Lambda_\tau(\sigma + \eta) - \log\!\Lambda_\tau(\sigma) -  D\log\!\Lambda_\tau(\sigma; \eta) \| +  C (\epsilon^{-1}\tau^{2\epsilon} + \tau^\epsilon\| \eta \|_{L^\infty(\Omega)} ).
\end{align}
The latter step is a consequence of Proposition \ref{prop:tau_to-zero} assuming that, say, $\| \eta \|_{L^\infty(\Omega)} \leq \essinf( \sigma/2)$ so that all conductivities considered during the limit process are uniformly bounded away from zero and infinity. In particular, the constant $C = C(\Omega,\sigma)>0$ in \eqref{eq:three_eps} is independent of (small enough) $\eta$ and $\tau>0$. 

Let $[\sigma, \sigma+\eta] := \{ \sigma +t\eta \mid t\in [0,1] \}$ denote the line segment in $L_+^\infty(\Omega)$ connecting $\sigma$ and $\sigma+\eta$. By virtue of Taylor's theorem for Banach spaces, Proposition~\ref{prop:logtau_2diff}, and the topology of $\mathscr{L}(L^2_\diamond(\partial\Omega))$ being finer than that of $\mathscr{L}(H_\diamond^{\epsilon}(\partial\Omega), H_\diamond^{-\epsilon}(\partial\Omega))$,
\begin{align*}
\| \log\!\Lambda_\tau(\sigma + \eta) - \log\!\Lambda_\tau(\sigma)  -  D\!\log\!\Lambda_\tau(\sigma; \eta) \| \hspace{-5.5cm}& \\[1mm] & \leq \frac{1}{2}  \sup_{\varsigma \in [\sigma, \sigma+\eta]}
\big\| D^2\!\log\!\Lambda_\tau (\varsigma; \, \cdot \, , \, \cdot \,) \big\|_{\mathscr{L}(L^\infty(\Omega)^2, \mathscr{L}(L^2_\diamond(\partial\Omega)))} \|\eta\|_{L^\infty(\Omega)}^2 
\leq C  \|\eta\|_{L^\infty(\Omega)}^2,
\end{align*}
where the constant $C = C(\Omega,\sigma)$ can be chosen independent of (small enough) $\eta$ and $\tau>0$. Combining this with \eqref{eq:three_eps}, we have altogether deduced that for any $\tau > 0$,
\begin{equation}
\| \log\!\Lambda(\sigma + \eta) - \log\!\Lambda(\sigma) -  D\!\log\!\Lambda(\sigma; \eta) \| \leq C  \big( \epsilon^{-1}\tau^{2\epsilon} + \tau^{\epsilon}\| \eta \|_{L^\infty(\Omega)} + \|\eta\|_{L^\infty(\Omega)}^2 \big). \label{eq:finaltauest}
\end{equation}
If one chooses $\tau = \tau(\|\eta\|_{L^\infty(\Omega)}) = \|\eta\|_{L^\infty(\Omega)}^{1/\epsilon}$, then the right-hand side of \eqref{eq:finaltauest} becomes $o(\|\eta\|_{L^\infty(\Omega)})$, and thereby the first part of the proof is complete.  

The next goal is to prove the continuity of $\sigma\mapsto D\!\log\!\Lambda(\sigma; \, \cdot \,)$ as a map from $L^\infty_+(\Omega)$ to $\mathscr{L}(L^\infty(\Omega), \mathscr{L}(L^2_\diamond(\partial\Omega)))$. According to the first part of the proof, $D\!\log\!\Lambda$ is the Fr\'echet derivative of $\log\!\Lambda$ in the topology of $\mathscr{L}(H^\epsilon_\diamond(\partial\Omega),H^{-\epsilon}_\diamond(\partial\Omega))$, and by repeating the first part in the proof of Proposition~\ref{prop:logtau_2diff} with $\tau =0$, its second Fr\'echet derivative is $D^2\!\log\!\Lambda := D^2F_0$.  The continuity of $\sigma\mapsto D\!\log\!\Lambda(\sigma; \, \cdot \,)$ now follows from any differentiable map being continuous, but let us anyway write a brief proof that provides an explicit Lipschitz-type estimate:

Obviously, $D^2\!\log\!\Lambda(\sigma;\eta,\, \cdot \,)$ is the Fr\'echet derivative of $\sigma\mapsto D\!\log\!\Lambda(\sigma;\eta)$ in the topology of $\mathscr{L}(L^2_\diamond(\partial\Omega))$  for any fixed $\eta\in L^\infty(\Omega)$.
Let us fix $\sigma_1,\sigma_2\in L_+^\infty(\Omega)$ and note that there exist scalars $\varsigma_+$ and $\varsigma_-$ such that $0 < \varsigma_- \leq \varsigma \leq \varsigma_+ < \infty$  for any $\varsigma \in [\sigma_1,\sigma_2] := \{ \sigma_1 + t(\sigma_2-\sigma_1) \mid t\in [0,1] \}$ almost everywhere in $\Omega$. Due to the mean-value theorem and Proposition~\ref{prop:logtau_2diff}, we have
\begin{align}
\| D\!\log\!\Lambda(\sigma_2;\eta) - D\!\log\!\Lambda(\sigma_1;\eta) \|_{\mathscr{L}(L^2_\diamond(\partial\Omega))} 
&\leq \sup_{\varsigma\in[\sigma_1,\sigma_2]} \big\| D^2\!\log\!\Lambda(\varsigma;\eta,\sigma_2-\sigma_1) \big\|_{\mathscr{L}(L^2_\diamond(\partial\Omega))} \nonumber \\
&\leq C(\Omega,\varsigma_-,\varsigma_+)\| \eta\|_{L^\infty(\Omega)} \|\sigma_2-\sigma_1\|_{L^\infty(\Omega)}. \label{eq:contDlam}
\end{align}
Taking the supremum over $\eta\in L^\infty(\Omega)$ with $\|\eta\|_{L^\infty(\Omega)} = 1$ implies $\sigma\mapsto D\!\log\!\Lambda(\sigma,\, \cdot \,)$ is locally Lipschitz continuous as a map $L^\infty_+(\Omega)\to\mathscr{L}(L^\infty(\Omega), \mathscr{L}(L^2_\diamond(\partial\Omega)))$.

Let us then consider \eqref{eq:maincont}. Due to the chain rule for Banach spaces, the Fr\'echet derivative of the completely logarithmic forward map considered in Theorem~\ref{thm:main} is given by
\begin{equation}
DL(\kappa; \eta) = D\!\log\!\Lambda({\rm e}^\kappa; \eta {\rm e}^\kappa), \qquad \kappa, \eta \in L^\infty(\Omega). \label{eq:DLrelation}
\end{equation}
In particular, it follows from Proposition~\ref{prop:logtau_diff} and \eqref{eq:contDlam} that for any $\kappa_1,\kappa_2 \in L^\infty(\Omega)$,
\begin{align*}
\| DL(\kappa_2;\eta) - DL(\kappa_1;\eta) \|_{\mathscr{L}(L^2_\diamond(\partial\Omega))} 
&= \big\| D\!\log\!\Lambda({\rm e}^{\kappa_2}; \eta {\rm e}^{\kappa_2}) - D\!\log\!\Lambda({\rm e}^{\kappa_1}; \eta {\rm e}^{\kappa_1}) \big\|_{\mathscr{L}(L^2_\diamond(\partial\Omega))} \\[1mm]
&\leq \big\| D\!\log\!\Lambda({\rm e}^{\kappa_2}; \eta {\rm e}^{\kappa_2}) - D\!\log\!\Lambda({\rm e}^{\kappa_1}; \eta {\rm e}^{\kappa_2}) \big\|_{\mathscr{L}(L^2_\diamond(\partial\Omega))} \\
&\phantom{{}\leq{}}+ \big\| D\!\log\!\Lambda({\rm e}^{\kappa_1}; \eta ({\rm e}^{\kappa_2}-{\rm e}^{\kappa_1})) \big\|_{\mathscr{L}(L^2_\diamond(\partial\Omega))} \\[1mm]
&\leq C \|\eta\|_{L^\infty(\Omega)} \big( \| {\rm e}^{\kappa_2}\|_{L^\infty(\Omega)}+1 \big) \| {\rm e}^{\kappa_2} - {\rm e}^{\kappa_1}\|_{L^\infty(\Omega)} \\[1mm]
&\leq C \|\eta\|_{L^\infty(\Omega)} \| {\rm e}^{\kappa_2} - {\rm e}^{\kappa_1}\|_{L^\infty(\Omega)},
\end{align*}
where $C>0$ only depends on $\Omega$ and $\max_{j=1,2} \| \kappa_j \|_{L^\infty(\Omega)}$. Taking the supremum over $\eta\in L^\infty(\Omega)$ with $\|\eta\|_{L^\infty(\Omega)}=1$ concludes the proof. \hfill $\square$

Finally, we consider the proofs of Corollary~\ref{cor:main} and Corollary~\ref{cor:DLformula}.

{\em Proof of Corollary~\ref{cor:main}}. Postponing the proof of \eqref{eq:Llocaluniformcont} till the end, clearly the rest of Corollary~\ref{cor:main} is equivalent to the difference $\log\!\Lambda(\sigma_2) - \log\!\Lambda(\sigma_1)$ continuously extending to a self-adjoint operator in $\mathscr{L}(L^2_\diamond(\partial\Omega))$ for any $\sigma_1,\sigma_2\in L^\infty_+(\Omega)$. For a fixed $\epsilon\in (0,\tfrac{1}{2}]$ and $\sigma_1,\sigma_2\in L^\infty_+(\Omega)$, we define a function $G$, with the derivative $G'(t) := DG(t;1)$, via
\begin{alignat*}{2}
G &: \ t\mapsto \log\!\Lambda(\sigma_1 + t(\sigma_2-\sigma_1)),  && \qquad [0,1] \to \mathscr{L}(H^\epsilon_\diamond(\partial\Omega), H^{-\epsilon}_\diamond(\partial\Omega)), \\
G' \! &: \ t\mapsto D\!\log\!\Lambda\big(\sigma_1+t(\sigma_2-\sigma_1); \sigma_2-\sigma_1\big),  && \qquad [0,1] \to \mathscr{L}(H^\epsilon_\diamond(\partial\Omega), H^{-\epsilon}_\diamond(\partial\Omega)).
\end{alignat*}
In particular, $G'$ is continuous as a mapping to $\mathscr{L}(L^2_\diamond(\partial\Omega))$ due to \eqref{eq:contDlam}, and thus it is also continuous with respect to the coarser topology of $\mathscr{L}(H^\epsilon_\diamond(\partial\Omega), H^{-\epsilon}_\diamond(\partial\Omega))$. By virtue of the fundamental theorem of calculus for Bochner integrals on the real line, we thus have
\begin{equation}
\log\!\Lambda(\sigma_2) - \log\!\Lambda(\sigma_1) = G(1)-G(0) = \int_0^1 G'(t) \, {\rm d} t \label{eq:fundamentaltheoremcalculus}
\end{equation}
as an operator in $\mathscr{L}(H^\epsilon_\diamond(\partial\Omega), H^{-\epsilon}_\diamond(\partial\Omega))$. Since $t\mapsto \|G'(t)\|_{\mathscr{L}(L^2_\diamond(\partial\Omega))}$ is continuous by \eqref{eq:contDlam} and $G'(t)$ defines a  self-adjoint operator in $\mathscr{L}(L^2_\diamond(\partial\Omega))$ for all $t\in[0,1]$, the right-hand side of \eqref{eq:fundamentaltheoremcalculus} extends continuously to a self-adjoint operator in $\mathscr{L}(L^2_\diamond(\partial\Omega))$, providing the sought for extension.

Finally we prove \eqref{eq:Llocaluniformcont}, which is essentially equivalent to the difference $H(\sigma) := \log\!\Lambda(\sigma) - \log\!\Lambda(\sigma_0)$ being locally Lipschitz continuous as a map $L^\infty_+(\Omega) \ni \sigma \mapsto H(\sigma) \in \mathscr{L}(L^2_\diamond(\partial\Omega))$ for any fixed $\sigma_0\in L^\infty_+(\Omega)$. The proof of Theorem~\ref{thm:main} indicates that $D\!\log\!\Lambda = DF_0$ is the Fr\'echet derivative of $H$ in the topology of $\mathscr{L}(L^2_\diamond(\partial\Omega))$.
As in \eqref{eq:contDlam}, it thus follows by the mean-value theorem and Proposition~\ref{prop:logtau_diff} that
\begin{equation*}
	\| H(\sigma_2) - H(\sigma_1) \|_{\mathscr{L}(L^2_\diamond(\partial\Omega))} \leq C \|\sigma_2-\sigma_1\|_{L^\infty(\Omega)},
\end{equation*}
where $C>0$ only depends on $\Omega$ and scalars $\varsigma_+$ and $\varsigma_-$ satisfying $0 < \varsigma_- \leq \sigma_j\leq \varsigma_+ < \infty$ almost everywhere in $\Omega$ for $j \in \{1,2\}$. Choosing $\sigma_j = {\rm e}^{\kappa_j}$, $j=1,2$, for arbitrary $\kappa_1, \kappa_2 \in L^\infty(\Omega)$ leads to \eqref{eq:Llocaluniformcont} and completes the proof.~\hfill $\square$

{\em Proof of Corollary~\ref{cor:DLformula}}. This result is a direct consequence of \eqref{eq:DLrelation} and Proposition~\ref{prop:DLambdaformula} with $\tau=0$. \hfill $\square$

\section{Concluding remarks}
In order to prove the completely logarithmic forward map of EIT really exhibits a low degree of nonlinearity, one should establish (favorable) norm bounds for its second derivative; consider~\eqref{eq:second_derivative} with $\tau=0$, \eqref{eq:Lderiva} and the chain rule for Banach spaces. Such considerations are left for future studies.

Although this work only considered the conductivity equation motivated by the observations in~\cite{Hyvonen18}, we expect that similar differentiability results also hold for the logarithms of ND maps defined by other linear elliptic equations over bounded Lipschitz domains.

\subsection*{Acknowledgments}

This work was supported by the Academy of Finland (decision 312124) and the Aalto Science Institute (AScI).

\appendix

\section{Derivatives of the Neumann-to-Dirichlet map}
\label{app:derivatives}

The material presented in this appendix has intimate connections to \cite{Calderon80} and \cite[Appendix~B]{Garde17}, where the analytic dependence of the DN and ND maps on the conductivity is considered. 

As above, let $\Omega \subset \R^d$, $d \geq 2$, be a bounded Lipschitz domain. Throughout this section we utilize the following norm equivalence that is a straightforward consequence of the Poincar\'e inequality:
\begin{equation}
  \label{eq:normequi}
	\|\nabla u\|_{L^2(\Omega)} \leq \|u\|_{H^1(\Omega)/\mathbb{C}}:= \inf_{c \in \C} \| u - c \|_{H^1(\Omega)} \leq C(\Omega) \|\nabla u\|_{L^2(\Omega)}.
\end{equation}
This also demonstrates that $H^1(\Omega)/\mathbb{C}$ is a Hilbert space when equipped with the inner product
\begin{equation}
  \label{eq:innerp}
(u,v)_\sigma :=  \int_{\Omega} \sigma \nabla u \cdot \nabla \overline{v} \, {\rm d} x , \qquad u,v \in H^1(\Omega)/\C,
\end{equation}
for any fixed $\sigma\in L^\infty_+(\Omega)$.

Let us introduce two auxiliary operators, namely
\begin{align*}
  N(\sigma) \colon \left\{ \begin{array}{l} f \mapsto u, \\[1mm]
    H^{-1/2}_\diamond(\partial\Omega) \to  H^1(\Omega)/\mathbb{C},
  \end{array} \right. 
  \qquad
  P(\sigma, \eta) \colon \left\{ \begin{array}{l} \tilde{u} \mapsto w, \\[1mm]
    H^1(\Omega)/\mathbb{C} \to  H^1(\Omega)/\mathbb{C},
    \end{array} \right.
\end{align*}
where $u \in H^1(\Omega)/\mathbb{C}$ is the unique solution of \eqref{eq:varcont} and $w \in H^1(\Omega)/\mathbb{C}$ is the unique solution of
\begin{equation}
\label{eq:perturbop}
\int_\Omega \sigma \nabla w \cdot \nabla \overline{v} \, {\rm d} x = - \int_\Omega \eta \nabla \tilde{u} \cdot \nabla \overline{v} \, {\rm d} x \qquad \text {for all } v \in H^1(\Omega)/\C
\end{equation}
for given $\sigma \in L^\infty_+(\Omega)$, $\eta \in L^\infty(\Omega)$, and $\tilde{u} \in H^1(\Omega)/\C$. The unique solvability of \eqref{eq:perturbop} as well as the bound
\begin{equation}
\label{eq:Pbound}  
\|P(\sigma, \eta)\|_{\mathscr{L}(H^1(\Omega)/\mathbb{C})} \leq  \frac{C(\Omega)}{\essinf \sigma} \, \|\eta\|_{L^\infty(\Omega)} 
\end{equation}
follow by combining \eqref{eq:normequi} and \eqref{eq:innerp} with the Lax--Milgram lemma. The boundedness of $N(\sigma)$ is guaranteed by \eqref{eq:Ubnd}.

It turns out that all derivatives for the standard forward map $\sigma \mapsto \Lambda(\sigma)$ of EIT can be explicitly represented with the help of $N(\sigma)$, $P(\sigma, \eta)$, and the `nonstandard' trace map
\begin{equation*}
{\rm tr} \colon \left\{ \begin{array}{l} v \mapsto V, \\[1mm]
    H^1(\Omega)/\mathbb{C} \to H^{1/2}_\diamond(\partial \Omega),
  \end{array} \right. 
\end{equation*}
where, in the spirit of \eqref{eq:Udef}, $V \in  H^{1/2}_\diamond(\partial \Omega)$ is the unique zero-mean representative of the quotient equivalence class $v|_{\partial \Omega} \in  H^{1/2}(\partial \Omega) /\C$. It is straightforward to confirm that ${\rm tr}: H^1(\Omega)/\mathbb{C} \to H^{1/2}_\diamond(\partial \Omega)$ inherits boundedness from the standard trace map. Indeed, by using the definition of quotient norms, it follows that
\begin{equation*}
	\| v|_{\partial \Omega} \|_{H^{1/2}(\partial \Omega) /\C} \leq C(\Omega)
	\| v \|_{H^{1}(\Omega) /\C} \qquad \text{for all } v \in H^{1}(\Omega) /\C,
\end{equation*}
and $V \in H^{1/2}_\diamond(\partial \Omega)$ is precisely the member of the equivalence class $v|_{\partial \Omega} \in H^{1/2}(\partial \Omega) /\C$ that realizes the quotient norm on the left-hand side.

We start with a simple lemma on the differentiability of $P$.

\begin{lemma}
\label{lemma:perturbop}
The map $L^\infty_+(\Omega) \times L^\infty(\Omega) \ni (\sigma, \eta) \mapsto P(\sigma, \eta) \in \mathscr{L}(H^1(\Omega)/\mathbb{C})$ is continuous, and it is linear in the second variable. Furthermore, $P$ is Fr\'echet differentiable and its partial derivative with respect to the first variable, $D_\sigma P(\sigma, \eta; \, \cdot \, ) \in \mathscr{L}(L^\infty(\Omega), \mathscr{L}(H^1(\Omega)/\mathbb{C}))$, admits the representation
\begin{equation*}
	D_\sigma P(\sigma, \eta; \xi) = P(\sigma, \xi)P(\sigma, \eta)
\end{equation*}
for all $\sigma \in L^\infty_+(\Omega)$ and $\eta, \xi \in L^\infty(\Omega)$.
\end{lemma}

\begin{proof}
  It is obvious from \eqref{eq:perturbop} and \eqref{eq:Pbound} that the map $L^\infty(\Omega) \ni \eta \mapsto P(\sigma, \eta) \in \mathscr{L}(H^1(\Omega)/\mathbb{C})$ is linear and uniformly bounded over all $\sigma$ in any subset of $L^\infty_+(\Omega)$ that is bounded uniformly away from zero. As a consequence, the continuity and Fr\'echet differentiability of $L^\infty_+(\Omega) \times L^\infty(\Omega) \ni (\sigma, \eta) \mapsto P(\sigma, \eta) \in \mathscr{L}(H^1(\Omega)/\mathbb{C})$ follows from that of $L^\infty_+(\Omega) \ni \sigma \mapsto P(\sigma, \eta) \in \mathscr{L}(H^1(\Omega)/\mathbb{C})$ for an arbitrary but fixed $\eta \in L^\infty(\Omega)$.

 Due to the definition of $P$ based on \eqref{eq:perturbop}, the following identities hold for all $\tilde{u}, v \in H^1(\Omega)/\mathbb{C}$, $\sigma \in L^\infty_+(\Omega)$, and $\xi \in L^\infty(\Omega)$ satisfying $\sigma + \xi \in L^\infty_+(\Omega)$: 
	\begin{equation*}
	\int_\Omega (\sigma + \xi) \nabla P(\sigma + \xi, \eta) \tilde{u} \cdot \nabla \overline{v} \, {\rm d} x = - \int_\Omega \eta \nabla \tilde{u} \cdot \nabla \overline{v} \, {\rm d}x  =
	\int_\Omega \sigma \nabla P(\sigma, \eta) \tilde{u} \cdot \nabla \overline{v} \, {\rm d} x. 
	\end{equation*}
Hence,
	\begin{equation*}
		\int_\Omega \sigma  \nabla \big(P(\sigma + \xi, \eta) - P(\sigma, \eta)\big)\tilde{u} \cdot \nabla \overline{v} \, {\rm d} x 
		= - \int_\Omega \xi \nabla P(\sigma+\xi, \eta) \tilde{u} \cdot \nabla \overline{v} \, {\rm d} x,
	\end{equation*}
which by \eqref{eq:perturbop} means that 
\begin{equation}
	P(\sigma + \xi, \eta) - P(\sigma, \eta) = P(\sigma,\xi)P(\sigma+\xi, \eta). \label{eq:Pdifference}
\end{equation}
In particular, by virtue of \eqref{eq:Pdifference} and \eqref{eq:Pbound},
\begin{equation}
  \label{eq:Pcont}
  \| P(\sigma+\xi, \eta) - P(\sigma, \eta) \|_{\mathscr{L}(H^1(\Omega)/\mathbb{C})}
  \leq \frac{C(\Omega)}{\essinf (\sigma) \essinf (\sigma + \xi)} \|\xi\|_{L^\infty(\Omega)} \|\eta\|_{L^\infty(\Omega)} ,
\end{equation}
which proves the claim about continuity. Resorting to \eqref{eq:Pdifference} for a second time yields
\begin{align}
	  \|P(\sigma + \xi, \eta) - P(\sigma, \eta) - P(\sigma, \xi) P(\sigma, \eta) \|_{\mathscr{L}(H^1(\Omega)/\mathbb{C})} \hspace{-3cm}& \notag \\[1mm]
        &= \big\| P(\sigma, \xi) \big(P(\sigma+\xi, \eta) -  P(\sigma, \eta)\big)\big\|_{\mathscr{L}(H^1(\Omega)/\mathbb{C})} \notag \\
		&\leq \frac{C(\Omega)}{\essinf (\sigma^2) \essinf (\sigma + \xi)} \|\xi\|^2_{L^\infty(\Omega)} \|\eta\|_{L^\infty(\Omega)}, \label{eq:Pderivbnd}
\end{align}
where the last step follows by combining \eqref{eq:Pbound} and \eqref{eq:Pcont}. Since the right-hand side of \eqref{eq:Pderivbnd} is $o(\|\xi\|_{L^\infty(\Omega)})$, this completes the proof.        
\end{proof}

We next show that ${\rm tr}\, P(\sigma, \, \cdot \, ) N(\sigma) \in \mathscr{L}(L^\infty(\Omega), \mathscr{L}(H^{-1/2}_\diamond(\partial \Omega), H^{1/2}_\diamond(\partial \Omega)))$ is the Fr\'echet derivative of the standard forward map $L^\infty_+(\Omega) \ni \sigma \mapsto \Lambda(\sigma) = {\rm tr} \, N(\sigma) \in \mathscr{L}(H^{-1/2}_\diamond(\partial \Omega), H^{1/2}_\diamond(\partial \Omega))$. In fact, we introduce an explicit formula involving only ${\rm tr}$, $P(\sigma, \, \cdot \,)$, and $N(\sigma)$ for all derivatives of the standard forward map up to an arbitrary order. To this end, let $\rho_k$ be the collection of all permutations of indices up to $k\in \N$, i.e.
$$
{\rho_k = \{(\alpha_1, \ldots, \alpha_k) \ | \ \alpha_i \in \{1, \ldots, k\} \ \text{and} \ \alpha_i \neq \alpha_j \ \text{if} \ i \neq j\} }.
$$

\begin{theorem} \label{thm:NDdiff}
The standard forward map  $L^\infty_+(\Omega) \ni \sigma \mapsto \Lambda(\sigma) \in \mathscr{L}(H^{-1/2}_\diamond(\partial \Omega),H^{1/2}_\diamond(\partial \Omega))$ is infinitely times continuously Fr\'echet differentiable and its derivatives are defined by
\begin{equation}
\label{eq:Lderiva}
	D^k \! \Lambda(\sigma; \eta_1, \ldots, \eta_k) = \sum_{\alpha \in \rho_k} {\rm tr}\, P(\sigma, \eta_{\alpha_1})\ldots P(\sigma, \eta_{\alpha_k})N(\sigma), \qquad k \in \N,
\end{equation}
for  $\eta_1, \dots, \eta_k \in L^\infty(\Omega)$.
The standard forward map is also analytic:
	\begin{equation*}
		\Lambda(\sigma + \eta) = \sum_{k = 0}^\infty \frac{1}{k!} D^k \Lambda(\sigma; \eta, \ldots, \eta),
	\end{equation*}
	 for all $\sigma \in L^\infty_+(\Omega)$ and $\eta \in L^\infty(\Omega)$ such that $\sigma + \eta \in L^\infty_+(\Omega)$ and $\| P(\sigma, \eta) \|_{\mathscr{L}(H^1(\Omega)/\mathbb{C})} < 1$.
\end{theorem}

\begin{proof}
Since $\Lambda(\sigma) = {\rm tr} \, N(\sigma)$ and the map ${\rm tr}: H^1(\Omega)/\C \to H^{1/2}_\diamond(\partial \Omega)$ is linear, bounded, and independent of $\sigma \in L^\infty(\Omega)$, it suffices to prove that $L^\infty_+(\Omega) \ni \sigma \mapsto N(\sigma) \in \mathscr{L}(H^{-1/2}_\diamond(\partial \Omega),H^1(\Omega)/\C)$ is analytic and its derivatives are given as
\begin{equation}
\label{eq:Nderiva}
D^k \!  N(\sigma; \eta_1, \ldots, \eta_k) = \sum_{\alpha \in \rho_k} P(\sigma, \eta_{\alpha_1})\ldots P(\sigma, \eta_{\alpha_k})N(\sigma), \qquad k \in \N,
\end{equation}
for any $\eta_1, \dots, \eta_k \in L^\infty(\Omega)$. Observe that the formula \eqref{eq:Nderiva} clearly defines $D^k \!  N(\sigma; \,\cdot\, , \ldots, \,\cdot\,)$ as an element of $\mathscr{L}(L^\infty(\Omega)^k, \mathscr{L}(H^{-1/2}_\diamond(\partial \Omega), H^1(\Omega)/\C))$ for any $\sigma \in L^\infty_+(\Omega)$ due to \eqref{eq:Ubnd}, \eqref{eq:Pbound}, and the linearity of $P$ in its second variable.

Let $\sigma \in L^\infty_+(\Omega)$ be arbitrary and $\eta \in L^\infty(\Omega)$ such that $\sigma + \eta \in L^\infty_+(\Omega)$. To begin with, note that by virtue of \eqref{eq:varcont} and the definition of $N(\sigma)$,
\begin{equation*}
\int_\Omega (\sigma + \eta) \nabla N(\sigma + \eta) f \cdot \nabla\overline{v} \, {\rm d} x = \big\langle f, v|_{\partial \Omega} \big\rangle = \int_\Omega \sigma \nabla N(\sigma) f \cdot \nabla\overline{v} \, {\rm d} x
\end{equation*}
for all $f \in H^{-1/2}_\diamond(\partial \Omega)$ and $v \in H^1(\Omega)/\C$.
It thus follows from the definition of $P(\sigma, \eta)$ that
	\begin{align*}
		\int_\Omega \sigma \nabla P(\sigma, \eta) N(\sigma + \eta) f \cdot \nabla\overline{v} \, {\rm d} x =& -\int_\Omega \eta \nabla N(\sigma + \eta) f \cdot \nabla\overline{v} \, {\rm d} x\\
		=& \int_\Omega (\sigma + \eta) \nabla N(\sigma + \eta) f \cdot \nabla\overline{v} \, {\rm d} x - \int_\Omega \sigma \nabla N(\sigma) f \cdot \nabla\overline{v} \, {\rm d} x \\
		&- \int_\Omega \eta \nabla N(\sigma + \eta) f \cdot \nabla\overline{v} \, {\rm d} x \\
		=& \int_\Omega \sigma \nabla \big(N(\sigma + \eta) - N(\sigma)\big) f \cdot \nabla\overline{v} \, {\rm d} x
	\end{align*}
for all $v \in H^1(\Omega)/\mathbb{C}$ and $f \in H^{-1/2}_\diamond(\partial \Omega)$. Hence, it must hold (cf.~\eqref{eq:innerp}) that 
\begin{equation*}
P(\sigma, \eta) N(\sigma + \eta) = N(\sigma + \eta)  - N(\sigma).	
\end{equation*}
Rearranging this equality as $(I - P(\sigma, \eta))N(\sigma + \eta) = N(\sigma)$ and requiring $\|\eta\|_{L^\infty(\Omega)}$ to be small enough to guarantee ${\| P(\sigma, \eta) \|_{\mathscr{L}(H^1(\Omega)/\mathbb{C})} < 1}$ (cf.~\eqref{eq:Pbound}), we may employ a Neumann series to write
\begin{equation*}
N(\sigma + \eta) = \sum_{k=0}^\infty P(\sigma, \eta)^k N(\sigma).
\end{equation*}
This proves the analyticity of $\sigma\mapsto N(\sigma)$ and, in particular, shows that $P(\sigma, \, \cdot \,) N(\sigma)$ is indeed its Fr\'echet derivative.

Using the product rule for Banach spaces and Lemma~\ref{lemma:perturbop}, the second Fr\'echet derivative of $\sigma\mapsto N(\sigma)$ can be written as 
	\begin{equation*}
		D^2 \! N(\sigma; \eta, \xi) = P(\sigma, \xi) P(\sigma, \eta) N(\sigma) + P(\sigma, \eta) P(\sigma, \xi) N(\sigma)
	\end{equation*}
 for all $\sigma \in L^\infty_+(\Omega)$ and $\eta, \xi \in L^\infty(\Omega)$.
The formula \eqref{eq:Nderiva} for an arbitrary $k \in \N$ then follows by recursively applying the product rule. The continuity of the derivatives is an immediate consequence of their Fr\'echet differentiability.
\end{proof}

As $\Lambda(\sigma): H^{-1/2}_\diamond(\partial \Omega) \to H^{1/2}_\diamond(\partial \Omega)$ is symmetric with respect to the dual bracket, the same also holds for all its derivatives $D^k \!\Lambda(\sigma; \eta_1, \ldots, \eta_k): H^{-1/2}_\diamond(\partial \Omega) \to H^{1/2}_\diamond(\partial \Omega)$ for any $k \in \N$ and all $\eta_1, \dots, \eta_k \in L^\infty(\Omega)$. Indeed, it is easy to check by recursively employing the definition of Fr\'echet differentiability that the symmetric part $\tfrac{1}{2} ( D^k \!\Lambda(\sigma; \eta_1, \ldots, \eta_k) + D^k \!\Lambda(\sigma; \eta_1, \ldots, \eta_k)^* )$ also defines a $k$th derivative for $\sigma \mapsto \Lambda(\sigma)$. Hence, any non-symmetry of $D^k \!\Lambda(\sigma; \eta_1, \ldots, \eta_k)$ would contradict the uniqueness of Fr\'echet derivatives.

We complete this appendix by presenting a corollary that covers \eqref{eq:FrDB} and \eqref{eq:FrDB2} as special cases.

\begin{corollary} \label{coro:NDdiffbound}
For any $\sigma \in L^\infty_+(\Omega)$, it holds
	\begin{equation*}
		\|D^k\! \Lambda(\sigma; \,\cdot \,, \dots , \, \cdot \,) \|_{\mathscr{L}(L^\infty(\Omega)^k, \mathscr{L}(H^{-1/2}_\diamond(\partial \Omega), H^{1/2}_\diamond(\partial \Omega)))} \leq \frac{C}{\essinf \sigma^{k+1}},
	\end{equation*}
where $C = C(\Omega, k) > 0$ is independent of $\sigma$.
\end{corollary}

\begin{proof}
 The claim is an immediate consequence of \eqref{eq:Lderiva}, \eqref{eq:Ubnd}, \eqref{eq:Pbound}, and the boundedness of the nonstandard trace map ${\rm tr}: H^1(\Omega)/\C \to H^{1/2}_\diamond(\partial \Omega)$.
\end{proof}

\section{On equivalent norms for \texorpdfstring{$H_\diamond^r(\partial \Omega)$}{mean-free H\^{}r}}
\label{app:norms}
Let $\{ \lambda_k(\sigma), \phi_k(\sigma) \}_{k \in \N}$ be a normalized eigensystem for the ND operator $\Lambda(\sigma): H_\diamond^{-1/2}(\partial \Omega) \to  H_\diamond^{1/2}(\partial \Omega)$, with $\sigma \in L^\infty_+(\Omega)$ and a bounded Lipschitz domain $\Omega$. Consult Section~\ref{sec:param} for more detailed definitions of these entities.
Let us start by introducing the (unbounded for $r<0$) powers of $\Lambda(\sigma): L^2_\diamond(\partial \Omega) \to L^2_\diamond(\partial \Omega)$ defined via
\begin{equation}
\label{eq:Lpower}
\Lambda^{2r}(\sigma): f \mapsto \sum_{k=1}^\infty \lambda_k^{2r}(\sigma) \,\langle f, \phi_k(\sigma) \rangle \, \phi_k(\sigma)
\end{equation}
for $-\tfrac{1}{2} \leq r \leq \tfrac{1}{2}$. 

\begin{proposition}
	\label{prop:power_map}
	The operator $\Lambda^{2r}(\sigma)$ defined by \eqref{eq:Lpower} can be interpreted as a symmetric isomorphism from $H^{-r}_\diamond(\partial \Omega)$ to  $H^{r}_\diamond(\partial \Omega)$ for any $-\tfrac{1}{2} \leq r \leq \tfrac{1}{2}$.
\end{proposition}

\begin{proof}
	As indicated,~e.g.,~in the proof of \cite[Lemma~1, Appendix~A]{Hyvonen18}, the operator $\Lambda^{r}(\sigma)$, $0 \leq r \leq \tfrac{1}{2}$, is an isomorphism from $L^2_\diamond(\partial \Omega)$ to $H^{r}_\diamond(\partial \Omega)$ with the inverse $\Lambda^{-r}(\sigma)$. It is straightforward to check that the isomorphic dual operator $(\Lambda^{r}(\sigma))^*: H^{-r}_\diamond(\partial \Omega) \to L^2_\diamond(\partial \Omega)$ is also defined by \eqref{eq:Lpower} and, in particular, coincides with $\Lambda^{r}(\sigma)$ on $L^2_\diamond(\partial \Omega)$. Hence,
	\begin{equation*}
	\Lambda^{r}(\sigma) (\Lambda^{r}(\sigma))^*: f \mapsto  \sum_{k=1}^\infty \lambda_k^{2r}(\sigma) \,\langle f, \phi_k(\sigma) \rangle  \, \phi_k(\sigma)
	\end{equation*}
	is an isomorphism between $H^{-r}_\diamond(\partial \Omega)$ and $H^{r}_\diamond(\partial \Omega)$ for any $0 \leq r \leq \tfrac{1}{2}$ and it also obviously coincides with $\Lambda^{2r}(\sigma)$ on $L^2_{\diamond}(\partial \Omega)$, thus providing the sought for extension. To complete the proof, the claim for $-\tfrac{1}{2} \leq r \leq 0$ follows by simply considering the inverse of the isomorphic extension $\Lambda^{-2r}(\sigma) := \Lambda^{-r}(\sigma) (\Lambda^{-r}(\sigma))^*:  H^{r}_\diamond(\partial \Omega) \to H^{-r}_\diamond(\partial \Omega)$ constructed above.
\end{proof}

In the following, we drop the `dual star notation' and write any power of the Neumann-to-Dirichlet map as $\Lambda^{r}(\sigma)$ independently of its domain of definition that should be clear from the context. In particular, note that all of these powers are defined by \eqref{eq:Lpower}.

According to \cite[Lemma~1, Appendix~A]{Hyvonen18} and the remark preceding it, there exist constants $c,C>0$, depending only on $\sigma$, $\Omega$, and $-1/2\leq r \leq 1/2$, such that
\begin{equation}
\label{eq:norm_equi}
c \| f \|_{r,\sigma} \leq \| f \|_{H^r(\partial \Omega)} \leq C \| f \|_{r, \sigma} \qquad {\rm for} \ {\rm all} \ f \in H_\diamond^r(\partial \Omega),
\end{equation}
where the equivalent norm for the mean-free Sobolev space $H_\diamond^r(\partial \Omega)$ is defined via
\begin{equation}
\label{eq:s_sigma_norm}
\| f \|_{r,\sigma}^2 := \sum_{k=1}^\infty \lambda_k^{-2r}(\sigma) |\langle f, \phi_k(\sigma) \rangle |^2, \qquad -1/2\leq r \leq 1/2.
\end{equation}
The main result of this appendix essentially states that the constants in \eqref{eq:norm_equi} can be chosen to only depend on $\essinf \sigma$ and $\esssup \sigma$, not on $\sigma$ itself.

\begin{theorem}
	\label{thm:norm_equi}
	Let $-1/2\leq r \leq 1/2$ and $0 < \varsigma_- \leq \varsigma_+ < \infty$ be fixed. Then all norms $\| \cdot \|_{\sigma, r}: H_\diamond^r(\partial \Omega) \to \overline{\R_+}$ defined by \eqref{eq:s_sigma_norm} with some $\sigma \in L^\infty_+(\Omega)$ satisfying
	\begin{equation*}
	\varsigma_- \leq \sigma \leq \varsigma_+ \qquad {\rm a.e.} \ {\rm in} \ \Omega
	\end{equation*} 
	are jointly equivalent in the sense that \eqref{eq:norm_equi} holds with constants $c = c(\Omega,r, \varsigma_-, \varsigma_+)>0$ and $C = C(\Omega,r, \varsigma_-, \varsigma_+)>0$ independent of the actual conductivity  $\sigma$.
\end{theorem}

\begin{proof}
	To summarize, the assertion is a consequence of two results, with the first being the L\"owner--Heinz inequality \cite[Satz~3,~p.~426]{Heinz51} stating that $t\mapsto t^\alpha$ is \emph{operator monotone} for $\alpha\in[0,1]$; see \cite{Pedersen_1972} for a short proof. The second result is the monotonicity relation \cite{Ikehata98,Kang97}:
	\begin{equation}
	\label{eq:monot}
	\langle f, \Lambda(\sigma_2) f \rangle \leq \langle f, \Lambda(\sigma_1) f \rangle \qquad {\rm for} \ {\rm all} \  f \in H^{-1/2}_\diamond(\partial \Omega)
	\end{equation}
	if $\sigma_1 \leq \sigma_2$ almost everywhere in $\Omega$.

	Let us first consider the case $-1/2 \leq r \leq 0$. By applying the L\"owner--Heinz inequality with the power $0 \leq -2r \leq 1$ to the self-adjoint operators $\Lambda(\varsigma_+), \Lambda(\sigma), \Lambda(\varsigma_-): L^2_\diamond(\partial \Omega) \to L^2_\diamond(\partial \Omega)$, we get
	\begin{align}
	\label{eq:monot_frac}
	\langle f, \Lambda^{-2r}(\varsigma_+) f \rangle
	\leq \langle f, \Lambda^{-2r}(\sigma) f \rangle \leq  \langle f, \Lambda^{-2r}(\varsigma_-) f \rangle \qquad {\rm for} \ {\rm all} \ f \in L^2_\diamond(\partial \Omega).
	\end{align}
	By the density of the continuous embedding $L^2_\diamond(\partial \Omega) \hookrightarrow H^{r}_\diamond(\partial \Omega)$ and the boundedness of $\Lambda^{-2r}(\varsigma_+)$, $\Lambda^{-2r}(\sigma)$, $\Lambda^{-2r}(\varsigma_-): H^r_\diamond(\partial \Omega) \to H^{-r}_\diamond(\partial \Omega)$ guaranteed by Proposition~\ref{prop:power_map}, the inequality \eqref{eq:monot_frac} holds, in fact, for all $f \in H^r_\diamond(\partial \Omega)$.
	In particular, \eqref{eq:monot_frac} is just another way of writing
	\begin{equation*}
	\| f \|_{r,\varsigma_+}^2 \leq \| f \|_{r,\sigma}^2 \leq \| f \|_{r,\varsigma_-}^2
	\qquad {\rm for} \ {\rm all} \ f \in H^{r}_\diamond(\partial \Omega).
	\end{equation*}
	Hence, for all $f \in H^{r}_\diamond(\partial \Omega)$,
	\begin{equation}
	\label{eq:norm_equi_minus}
	c(\varsigma_-) \| f \|_{r,\sigma} \leq c(\varsigma_-) \| f \|_{r,\varsigma_-} \leq \| f \|_{H^r(\partial \Omega)} \leq C(\varsigma_+) \| f \|_{r,\varsigma_+} \leq
	C(\varsigma_+)  \| f \|_{r,\sigma},
	\end{equation}
	where the constants $c(\varsigma_-) =  c(\Omega,\varsigma_-,r) > 0$ and $C(\varsigma_+) = C(\Omega, \varsigma_+,r)$ correspond to the homogeneous conductivities $\varsigma_-$ and $\varsigma_+$ in \eqref{eq:norm_equi}, respectively. This completes the proof for $-1/2 \leq r \leq 0$.
	
	Observe that \eqref{eq:monot_frac} induces the `inverse estimate'
	\begin{equation}
	\label{eq:inverse}
	\langle \Lambda^{-2r}(\varsigma_-) f, f \rangle
	\leq \langle \Lambda^{-2r}(\sigma) f,f \rangle \leq  \langle \Lambda^{-2r}(\varsigma_+) f, f \rangle \qquad {\rm for} \ {\rm all} \ f \in H^{r}_\diamond(\partial \Omega)
	\end{equation}
	and $0 \leq r \leq 1/2$; the proof of this fact is included for completeness as Lemma~\ref{lemma:inverse} below. In other words,
	\begin{equation*}
	\| f \|_{r,\varsigma_-}^2 \leq \| f \|_{r,\sigma}^2 \leq \| f \|_{r,\varsigma_+}^2
	\qquad {\rm for} \ {\rm all} \ f \in H^{r}_\diamond(\partial \Omega),
	\end{equation*}
	and thus the proof for $0 \leq r \leq 1/2$ can straightforwardly be completed by exchanging the roles of $\varsigma_-$ and $\varsigma_+$ in \eqref{eq:norm_equi_minus}. 
\end{proof}

We complete this appendix and the whole paper by presenting a lemma proving \eqref{eq:inverse}. This result could also be proved by directly employing monotonicity properties of the Dirichlet-to-Neumann operator and applying the L\"owner--Heinz inequality that remains valid for unbounded self-adjoint operators. Moreover, one could consider a map operating over a general Gelfand triple in place of $\Lambda^{2r}(\sigma)$ and the Sobolev spaces $H^{r}_\diamond(\partial \Omega) \hookrightarrow L^2_\diamond(\partial \Omega) \hookrightarrow H^{-r}_\diamond(\partial \Omega)$.

\begin{lemma}
	\label{lemma:inverse}
	Let $\sigma_1, \sigma_2 \in L^\infty_+(\Omega)$ be such that $\sigma_1 \leq \sigma_2$ almost everywhere in $\Omega$. Then
	\begin{equation}
	\label{eq:inverse2}   
	\langle \Lambda^{-2r}(\sigma_1) f, f \rangle
	\leq \langle \Lambda^{-2r}(\sigma_2) f, f \rangle
	\end{equation}
	for all $0\leq r \leq \tfrac{1}{2}$ and $f \in H^{r}_\diamond(\partial \Omega)$.
\end{lemma}

\begin{proof}
	Due to \eqref{eq:monot} and the L\"owner--Heinz inequality,
	\begin{equation}
	\label{eq:converse}
	\langle g, \Lambda^{2r}(\sigma_1) g\rangle
	\geq \langle g, \Lambda^{2r}(\sigma_2) g \rangle \qquad {\rm for} \ {\rm all} \ g \in H^{-r}_\diamond(\partial \Omega).
	\end{equation}
	Since $\Lambda^{-r}(\sigma_2): L^2_\diamond(\partial \Omega) \to H^{-r}_\diamond(\partial \Omega)$ is an isomorphism and coincides with its dual on $H^r_\diamond(\partial\Omega)$, via the substitution $g = \Lambda^{-r}(\sigma_2)w$ for $w\in L^2_\diamond(\partial\Omega)$, it follows that \eqref{eq:converse} is equivalent to
	\begin{equation}
	\label{eq:pos_identity}
	\big \langle w, \Lambda^{-r}(\sigma_2) \Lambda^{2r}(\sigma_1) \Lambda^{-r}(\sigma_2) w \big\rangle \geq \| w \|_{L^2(\partial \Omega)}^2 \qquad {\rm for} \ {\rm all} \ w \in L^2_\diamond(\partial \Omega).
	\end{equation}
	In particular, $\Lambda^{-r}(\sigma_2) \Lambda^{2r}(\sigma_1) \Lambda^{-r}(\sigma_2): L^2_\diamond(\partial \Omega) \to L^2_\diamond(\partial \Omega)$ is a positive self-adjoint isomorphism, and thus it has a positive self-adjoint isomorphic square root $R: L^2_\diamond(\partial \Omega) \to  L^2_\diamond(\partial \Omega)$.
	
	For any $v \in L^2_\diamond(\partial \Omega)$, we may write
	\begin{align*}
	\big \langle \Lambda^{r}(\sigma_2) \Lambda^{-2r}(\sigma_1) \Lambda^{r}(\sigma_2) v, v \big \rangle =  \| R^{-1} v \|_{L^2(\partial \Omega)}^2 
	\leq
	\langle R^{-1} v,  R^2 R^{-1}v \rangle = \| v \|_{L^2(\partial \Omega)}^2,
	\end{align*}
	where the second step follows from \eqref{eq:pos_identity}. Employing the substitution $v = \Lambda^{-r}(\sigma_2) f$ for $f\in H^r_\diamond(\partial\Omega)$, this is equivalent to \eqref{eq:inverse2} and the proof is complete.
\end{proof}

\bibliographystyle{plain}
\bibliography{logL-refs}

\begin{thebibliography}{10}

\bibitem{Adler09}
A.~Adler, J.~H. Arnold, R.~Bayford, A.~Borsic, B.~Brown, P.~Dixon, T.~J.~C.
  Faes, I.~Frerichs, H.~Gagnon, Y.~G\"{a}rber, B.~Grychtol, G.~Hahn, W.~R.~B.
  Lionheart, A.~Malik, R.~P. Patterson, J.~Stocks, A.~Tizzard, N.~Weiler, and
  G.~K. Wolf.
\newblock {GREIT}:\ a unified approach to {2D} linear {EIT} reconstruction of
  lung images.
\newblock {\em Physiol. Meas.}, 30(6):S35--S55, 2009.

\bibitem{Borcea02}
L.~Borcea.
\newblock Electrical impedance tomography.
\newblock {\em Inverse Problems}, 18:R99--R136, 2002.

\bibitem{Borcea2002}
L.~Borcea.
\newblock Addendum to ``electrical impedance tomography".
\newblock {\em Inverse Problems}, 19:997--998, 2003.

\bibitem{Calderon80}
A.-P. {C}alder{\'o}n.
\newblock On an inverse boundary value problem.
\newblock In {\em Seminar on {N}umerical {A}nalysis and its {A}pplications to
  {C}ontinuum {P}hysics}, pages 65--73. Soc. Brasil. Mat., Rio de Janeiro,
  1980.

\bibitem{Cheney90}
M.~Cheney, D.~Isaacson, J.~C. Newell, S.~Simske, and J.~Goble.
\newblock {NOSER}: An algorithm for solving the inverse conductivity problem.
\newblock {\em Int. J. Imag. Syst. Tech.}, 2(2):66--75, 1990.

\bibitem{Cheney99}
M.~Cheney, D.~Isaacson, and J.C. Newell.
\newblock Electrical impedance tomography.
\newblock {\em SIAM Rev.}, 41:85--101, 1999.

\bibitem{Cheng89}
K.-S. Cheng, D.~Isaacson, J.~C. Newell, and D.~G. Gisser.
\newblock Electrode models for electric current computed tomography.
\newblock {\em IEEE Trans. Biomed. Eng.}, 36(9):918--924, 1989.

\bibitem{Dautray88}
R.~Dautray and J.~L. Lions.
\newblock {\em Mathematical Analysis and Numerical Methods for Science and
  Technology: Functional and Variational Methods}, volume~2.
\newblock Springer-Verlag, 1988.

\bibitem{Garde17}
H.~Garde and S.~Staboulis.
\newblock Convergence and regularization for monotonicity-based shape
  reconstruction in electrical impedance tomography.
\newblock {\em Numer. Math.}, 135:1221--1251, 2017.

\bibitem{Gilliam_2009}
D.~S. Gilliam, T.~Hohage, X.~Ji, and F.~Ruymgaart.
\newblock The {F}r{\'{e}}chet derivative of an analytic function of a bounded
  operator with some applications.
\newblock {\em Int.\ J.\ Math.\ Math.\ Sci.}, 2009:1--17, 2009.

\bibitem{Grisvard85}
P.~Grisvard.
\newblock {\em Elliptic Problems in Nonsmooth Domains}.
\newblock Pitman, 1985.

\bibitem{Harrach10}
B.~Harrach and J.~K. Seo.
\newblock Exact shape-reconstruction by one-step linearization in electrical
  impedance tomography.
\newblock {\em SIAM J. Math. Anal.}, 42(4):1505--1518, 2010.

\bibitem{Heinz51}
E.~Heinz.
\newblock Beitr\"age zur st\"orungstheorie der spektralzerlegung.
\newblock {\em Math. Ann.}, 123:415--438, 1951.

\bibitem{Hille_1952}
E.~Hille.
\newblock Une g\'{e}n\'{e}ralisation du probl\`eme de {C}auchy.
\newblock {\em Ann. Inst. Fourier Grenoble}, 4:31--48, 1952.

\bibitem{Hille_1957}
E.~Hille and R.~S. Phillips.
\newblock {\em Functional analysis and semi-groups}.
\newblock American Mathematical Society Colloquium Publications, vol. 31.
  American Mathematical Society, Providence, R. I., 1957.
\newblock rev. ed.

\bibitem{Hyvonen18}
N.~Hyv{\"o}nen and L.~Mustonen.
\newblock Generalized linearization techniques in electrical impedance
  tomography.
\newblock {\em Numer. Math.}, 140:95--120, 2018.

\bibitem{Ikehata98}
M.~Ikehata.
\newblock Size estimation of inclusion.
\newblock {\em J. Inverse Ill-Posed Problems}, 6:127–140, 1998.

\bibitem{Kang97}
H.~Kang, J.~Seo, and D.~Sheen.
\newblock The inverse conductivity problem with one measurement: Stability and
  estimation of size.
\newblock {\em SIAM J. Math. Anal.}, 28:1389--1405, 1997.

\bibitem{Lee1989}
J.~Lee and G.~Uhlmann.
\newblock Determining anisotropic real-analytic conductivities by boundary
  measurements.
\newblock {\em Comm. Pure Appl. Math.}, 42(2):1097--1112, 1989.

\bibitem{Pedersen_1972}
G.~K. Pedersen.
\newblock Some operator monotone functions.
\newblock {\em Proc.\ Amer.\ Math.\ Soc.}, 36(1):309--310, 1972.

\bibitem{Pedersen00}
G.~K. Pedersen.
\newblock Operator differentiable functions.
\newblock {\em Publ. RIMS, Kyoto Univ.}, 36:139--157, 2000.

\bibitem{Somersalo92}
E.~Somersalo, M.~Cheney, and D.~Isaacson.
\newblock Existence and uniqueness for electrode models for electric current
  computed tomography.
\newblock {\em SIAM J. Appl. Math.}, 52(4):1023--1040, 1992.

\bibitem{Uhlmann09}
G.~Uhlmann.
\newblock Electrical impedance tomography and {C}alder\'{o}n's problem.
\newblock {\em Inverse Problems}, 25(12):123011, 2009.

\end{thebibliography}

\end{document}